\documentclass[12pt, twoside]{article}
\usepackage{amsmath,amsthm,amssymb}
\usepackage{times}
\usepackage{enumerate}
\usepackage{color}
\usepackage{epstopdf,graphics,graphicx}

\theoremstyle{definition}
\newtheorem{thm}{Theorem}[section]

\newtheorem{lem}[thm]{Lemma}

\newtheorem{rem}[thm]{Remark}

\numberwithin{equation}{section}

\newcommand{\subjclass}[1]{\bigskip\noindent\emph{2010 Mathematics Subject Classification:}\enspace#1}
\newcommand{\keywords}[1]{\noindent\emph{Keywords:}\enspace#1}

\newcommand {\R}	{\mathbb{R}}
\newcommand {\N}	{\mathbb{N}}

\DeclareMathOperator{\Id}{Id}

\newcommand\bff{{\mathbf f}}
\newcommand\bfg{{\mathbf g}}

\newcommand\bfn{{\mathbf n}}

\newcommand\bfu{{\mathbf u}}
\newcommand\bfv{{\mathbf v}}

\newcommand\bfx{{\mathbf x}}

\newcommand\bfA{{\mathbf A}}

\newcommand\bfH{{\mathbf H}}
\newcommand\bfvecH{\vec{\mathbf H}}
\newcommand\bfI{{\boldsymbol I}}

\newcommand\bfM{{\mathbf M}}

\newcommand\bfpi{{\boldsymbol{\pi}}}

\newcommand\qand{\quad\hbox{ and }\quad}
\newcommand\qfor{\quad\hbox{ for }\quad}

\renewcommand{\d}{\text{d}}

\newcommand{\Ga}{\Gamma}

\newcommand{\mat}{\partial^{\bullet}}

\newcommand{\diff}{\frac{\d}{\d t}}
\newcommand{\nb}{\nabla}

\newcommand{\pa}{\partial}

\newcommand{\spn}{\textnormal{span}}

\def \t {(t)}

\def \to {\rightarrow}
\newcommand{\vphi}{\varphi}

\newcommand{\phiv}{\varphi^v}
\newcommand{\phin}{\varphi^\n}
\newcommand{\phiH}{\varphi^H}

\newcommand{\xs}{\bfx^\ast}

\newcommand{\vecH}{\vec{H}}

\newcommand{\n}{\nu}

\newcommand{\dimR}{n}
\newcommand{\dimGa}{m}
\newcommand{\dof}{N}

\begin{document}


\baselineskip=17pt


\title{A convergent finite element algorithm for mean curvature flow in arbitrary codimension}

\author{
	Tim Binz\\
	TU Darmstadt, Fachbereich Mathematik, \\ Schlossgartenstrasse 7, 64289 Darmstadt, Germany, \\
	binz@mathematik.tu-darmstadt.de\\[2mm]
	Bal\'{a}zs Kov\'{a}cs\\
	University of Regensburg, Faculty of Mathematics, \\ 93040 Regensburg, Germany \\
	balazs.kovacs@mathematik.uni-regensburg.de
}

\date{\today}

\maketitle


\begin{abstract}
Optimal-order uniform-in-time $H^1$-norm error estimates are given for semi- and full discretizations of mean curvature flow of surfaces in arbitrarily high codimension.
The proposed and studied numerical method is based on a parabolic system coupling the surface flow to evolution equations for the mean curvature vector and for the orthogonal projection onto the tangent space. The algorithm uses evolving surface finite elements and linearly implicit backward difference formulae. 
This numerical method admits a convergence analysis in the case of finite elements of polynomial degree at least two and backward difference formulae of orders two to five.
Numerical experiments in codimension 2 illustrate and complement our theoretical results.


\end{abstract}

\section{Introduction}

In this paper we prove semi- and fully discrete error bounds of a numerical algorithm for the evolution of a closed $\dimGa$-dimensional surface $\Ga\t \subset \R^\dimR$ evolving under mean curvature flow of \emph{arbitrary codimension}, with a particular interest for codimension at least $2$.

The key idea of the paper is to derive \emph{non-linear} parabolic evolution equations for the mean curvature vector $\vecH$ and the orthogonal projection onto the tangent space $\pi$ along the flow. Taking a similar approach as recent previous work for the numerical analysis of mean curvature flow \cite{MCF}, which first utilised such an approach using similar evolution equations for the (scalar) mean curvature and surface normal. 
For the numerical analysis of other geometric flows using this approach, see \cite{MCF_soldriven,Willmore,MCF_generalised}. 

Similarly as Huisken \cite{Huisken1984} did for codimension 1 mean curvature flow, in higher codimension \cite{AndrewsBaker} (or see \cite{Smoczyk_survey}) have derived numerous geometric evolution equations for various geometric quantities. 
For mean curvature flow in arbitrary codimension (and dimension), we derive here evolution equations for the mean curvature vector and the orthogonal projection onto the tangent space, which (to our knowledge) were not yet known in the literature. Until the present work it was also not evident that these evolution equations for $\vecH$ and $\pi$ form a closed system that does not involve further geometric quantities.

\medskip
We give a brief overview of mean curvature flow in higher codimension:

In two papers \cite{Altschuler,AltschulerGrayson} Altschuler and Grayson have proved the first results for curves in $\R^3$, namely that unlike for planar curves singularities may occur in finite time.
Ambrosio and Soner \cite{AmbrosioSoner_levelset,AmbrosioSoner_1} have studied high codimension mean curvature flow using a level set approach.
Andrews and Baker \cite{AndrewsBaker} have proved that submanifolds sufficiently close to the round sphere smoothly collapse to round points in finite time, and have proved pinching estimates. They derive evolution equations for geometric quantities along the flow, similar to \cite{Huisken1984} in codimension 1, but have not derived the closed system of evolution equations derived and used in this paper.
These pinching estimates were greatly refined by Naff in \cite{Naff}.
Ancient solutions were recently studied by Lynch and Nguyen \cite{LynchNguyen}.
In the survey article \cite{Smoczyk_survey} Smoczyk presents results on: short-time existence and uniqueness, long-time existence and convergence, and singularities.
The survey type article by Wang \cite{Wang_MCF_codim} collects several theorems on regularity, global existence and convergence, see also \cite{Wang_lecturenotes}.

We also give a literature overview on numerical methods for curve shortening and mean curvature flow in \emph{codimension at least $2$} (only giving a brief outlook on other flows):

Dziuk \cite{Dziuk_CSF_1994} and Deckelnick and Dziuk \cite{DeckelnickDziuk1994}  have both proposed and analysed finite element algorithms for curve shortening flow for curves in possibly higher codimension, and have proved semi-discrete $L^\infty(L^2)$- and $L^2(H^1)$-norm error estimates.
Carlini, Falcone and Ferretti \cite{CarliniFalconeFerretti_CSF_codim_2} proposed a semi-Lagrangian scheme for curve shortening flow in codimension $2$ (i.e., closed curves in $\R^3$), and have analysed its (conditional) consistency.
Pozzi proposed a numerical method for anisotropic curve shortening and mean curvature flow in higher codimension in \cite{Pozzi_CSF_codim,Pozzi_surf_codim}, and proved semi-discrete error estimates for curves in arbitrary codimension.
Barrett, Garcke and N\"urnberg \cite{BGN_curve_grad_flow} proposed numerical algorithms -- allowing tangential movements -- for gradient flows (including curve shortening and Willmore flow) for closed curves in $\R^\dimR$ ($\dimR\geq2$). In \cite{BGN_curves_2012} they discretized high-order flows for plane and space curves.
D\"orfler and N\"urnberg \cite{DoerflerNuernberg} have proposed finite element discretizations for gradient flows for general curvature energies of space curves. 
A tangentially redistributing scheme for 3-dimension curve evolutions was proposed in \cite{MikulaUrban}.

Apart from the convergence results by Dziuk \cite{Dziuk_CSF_1994} and Pozzi \cite{Pozzi_CSF_codim}, both results for \emph{curves in $\R^\dimR$}, we are not aware of any convergence results for mean curvature flow in \emph{high codimension}.

\medskip
The newly derived non-linear geometric evolution equations for the mean curvature vector $\vecH$ and orthogonal projection $\pi$ are coupled to the velocity law $v = \vecH$ and the ordinary differential equation (ODE) $\dot X = v \circ X$ describing the surface evolution. This geometric coupled system is then discretized using evolving surface finite elements (of degree at least 2) and using linearly implicit backward difference formulae (of order 2 to 5), under a mild step size restriction. 

We prove optimal-order time-uniform $H^1$-norm semi- and fully discrete error estimates for mean curvature flow in arbitrary codimension and dimension, utilising the newly derived geometric coupled system, for the surface position $X$ and the velocity $v$, and the geometric quantities $\vecH, \pi$. 
The functional-analytic setting for the spatial semi-discrete coupled system of mean curvature flow in arbitrary codimension is fundamentally different from the one for mean curvature flow \cite{MCF}. Still, the matrix--vector formulations of their respective semi-discretisations formally coincide. We regard this as an advantage of our algorithm.
They both use the same mass and stiffness matrices with different block-sizes, but the non-linear terms are more complicated than in \cite{MCF}, however both are locally Lipschitz continuous. 
Due to this purely formal analogy of the matrix--vector formulations, the convergence proofs for arbitrary codimension mean curvature flow are also formally coinciding with the respective proofs in \cite{MCF} for mean curvature flow. 
More precisely, since the non-linear terms are locally Lipschitz, the stability proofs of \cite{MCF} directly apply to the present case as well. Consistency proofs are shown using similar arguments. 

Arguably, for \emph{curves} the algorithm proposed here is more complicated to implement than the methods of Dziuk \cite{Dziuk_CSF_1994} and Deckelnick and Dziuk \cite{DeckelnickDziuk1994}, however, the algorithm proposed here comes with a convergence analysis for \emph{surfaces}.

The paper is organised as follows:
Section~\ref{section:evolution equations} 
introduces basic notations for arbitrary codimension submanifolds, and contains the main technical results of the paper: deriving the evolution equations for $\vecH$ and $\pi$.
Section~\ref{section:ESFEM} 
contains the evolving surface finite element spatial semi-discretization, and the corresponding matrix--vector formulation, discussing its relation to the matrix--vector formulation of mean curvature flow \cite{MCF}.
Section~\ref{section:BDF} 
presents the linearly implicit backward differentiation formulas.
Section~\ref{section:main results} 
contains the main results of the paper, semi- and fully discrete error bounds.
Section~\ref{section:numerics} 
reports on a large number of numerical experiments, to illustrate and complement our theoretical results, including convergence tests, comparisons with Dziuk's algorithm \cite{Dziuk_CSF_1994}, and present some examples from the literature.

\section{Evolution equations for mean curvature flow}
\label{section:evolution equations}

\subsection{Basic notions and notation}
\label{subsection:basic notions}

We start by introducing some basic concepts and notations. 

We consider an evolving $\dimGa$-dimensional (with $\dimGa = 1,2,3$) closed submanifold $\Ga[X] \subset \R^\dimR$, in other words $\Ga[X]$ is an $\dimGa$-dimensional submanifold in $\R^\dimR$ of codimension $\dimR-\dimGa$. In this paper we allow submanifolds of arbitrary codimension $\dimR-\dimGa \geq 1$.

The $\dimGa$-dimensional submanifold $\Ga[X]$ is given as the image
$$
\Ga[X] = \Ga[X(\cdot,t)] = \{ X(p,t) \,:\, p \in \Ga^0 \},
$$
of a smooth mapping $X \colon \Ga^0\times [0,T]\to \R^\dimR$ of an initial submanifold $\Ga^0$ such that $X(\cdot,t)$ is an embedding for every $t$, and $X(p,0)=p$. Here, the initial submanifold $\Ga^0 \subset \R^\dimR$ is smooth and of dimension $m$. 
In view of the subsequent numerical discretization, it is convenient to think of $X(p,t)$ as the position at time $t$ of a moving particle with label $p$, and of $\Ga[X] $ as a collection of such particles. 
This approach is similar as in \cite{MCF}.

The {\it velocity} $v(x,t)\in\R^\dimR$ at a point $x=X(p,t)\in\Ga[X(\cdot,t)]$  equals
\begin{equation}
\label{eq:velocity ODE}
\partial_t X(p,t)= v(X(p,t),t).
\end{equation}
For a known velocity field $v$, the position $X(p,t)$ at time $t$ of the particle with label $p$ is obtained by solving the ordinary differential equation \eqref{eq:velocity ODE} from $0$ to $t$ for a fixed $p$.

For a function $u(x,t)$ ($x\in \Ga[X]$, $0\le t \le T$) we denote the {\it material derivative} (with respect to the parametrization $X$) as
$$
\mat u(x,t) = \frac \d{\d t} \,u(X(p,t),t) \quad\hbox{ for } \ x=X(p,t).
$$

On a regular submanifold we denote by $g_{ij} = \langle \partial_i X, \partial_j X \rangle$ ($i,j=1,\dotsc,\dimGa$) the induced metric, and by $(g^{ij})$ its inverse.
Moreover, we denote by
\begin{equation*}
A(x) = (A_{ij}(x))_{i,j=1}^\dimGa
= \big( (\partial_i \partial_j X(p,t))^\bot \big)_{i,j=1}^\dimGa = \big( \partial_i \partial_j X - \Ga_{ij}^k \partial_k X \big)_{i,j=1}^\dimGa 
\in (\R^\dimR)^{m \times m} 
\end{equation*}
the second fundamental form. Here $\,^\bot$ denotes the orthogonal projection to the orthogonal complement of the tangent space of $\Ga[X]$ at $x = X(p,t)$.

Further, the \emph{mean curvature vector} is the trace of the Weingarten map, i.e.
\begin{equation*}
\vecH 
= g^{ij} A_{ij} =(g^{ij} \partial_i \partial_j X)^\bot = g^{ij} \partial_i \partial_j X - g^{ij} \Ga_{ij}^k \partial_k X
\in \R^\dimR ,
\end{equation*}
i.e.~we use the sign convention that for a sphere of radius $R$ the mean curvature vector $\vecH$ points inwards and has length $\dimGa/R$.

For every $x \in \Ga[X]$, we denote the \emph{orthogonal projection} from $\R^\dimGa$ to the tangent space of the submanifold $\Ga[X(\cdot,t)]$, at the point $x = X(p,t)$, by 
\begin{equation*}
\pi(x) 
= g^{ij} \partial_i X \otimes \partial_j X
\in \R^{\dimR \times \dimR} .
\end{equation*}


%

On any regular submanifold $\Ga\subset\R^\dimR$, the  {\it tangential gradient} $\nabla_{\Ga}u \colon \Ga\to\R^\dimR$ of a function $u \colon \Ga\to\R$ is given by $\nabla_{\Ga} u := g^{ij} \, \partial_i u \, \partial_j X$, and in the case of a vector-valued function $u=(u_1,\dotsc,u_\dimR)^T \colon \Ga\to\R^\dimR$, we define component-wise
$\nabla_{\Ga}u=
(\nabla_{\Ga}u_1, \dotsc, 
\nabla_{\Ga}u_\dimR)$, i.e. we use the convention that the gradient of $u$ has the gradient of the components as column vectors. 
We denote by $\varDelta_{\Ga} u=\nabla_{\Ga}\cdot \nabla_{\Ga}u$ the {\it Laplace--Beltrami operator} applied to $u$, so that on a closed surface $\int_\Ga \varDelta_\Ga u \, v = -\int_\Ga \nabla_{\Ga} u \cdot \nabla_{\Ga} v$, cf.~\cite{DziukElliott_acta}. 

%
%
%
%
%
%

\subsection{Evolution equations for orthogonal projection and mean curvature vector of a submanifold evolving under mean curvature flow}

Mean curvature flow (in arbitrary codimension) sets the velocity \eqref{eq:velocity ODE} of the submanifold $\Ga[X]$ to 
\begin{equation}
\label{eq:velocity law}
v = \vecH .
\end{equation}

For geometric surface flows, e.g.~see \cite{Huisken1984}, it is known that the geometric quantities satisfy evolution equations along the flow. The algorithm here is based on parabolic partial differential equations for the mean curvature vector $\vecH$ and the projection $\pi$, derived in the following result.

\begin{lem}
	\label{lemma:evolution equations}
	For a regular $\dimGa$-dimensional submanifold $\Ga[X] \subset \R^\dimR$ moving under mean curvature flow in codimension $\dimR - \dimGa$, the orthogonal projection $\pi$ and the mean curvature vector $\vecH$ satisfy
	\begin{subequations}
		\begin{align}
		\label{eq:evolution eq - pi}
		\mat \pi 
		&= \varDelta_{\Ga[X]} \pi 
		+ f_1(\pi), \\
		\label{eq:evolution eq - H}
		\mat \vecH 
		&= \varDelta_{\Ga[X]} \vecH + f_2(\pi,\vecH),
		\end{align} 
	\end{subequations}
	where the Laplace--Beltrami operator is understood componentwise. The non-linear terms are given componentwise, for $\alpha,\beta=1,\dotsc,\dimR$, by
	\begin{equation}
	\label{eq:non-linear expressions - f_1 and f_2}
	\begin{aligned} 
	f_1(\pi)_{\alpha \beta} &= 2 \sum\limits_{\mu=1}^{\dimR} \nabla_{\Ga[X]} \pi_{\alpha \mu} \cdot \nabla_{\Ga[X]} \pi_{\beta \mu}
	- 4 \sum\limits_{\mu,\kappa=1}^{\dimR} \pi_{\mu\kappa} \nabla_{\Ga[X]} \pi_{\alpha\mu} \cdot  \nabla_{\Ga[X]} \pi_{\beta\kappa}, \\ 
	f_2(\pi,H)_{\alpha} &= 2 \sum\limits_{\mu=1}^{\dimR} \nabla_{\Ga[X]} \pi_{\alpha\mu} \cdot \nabla_{\Ga[X]} H_{\mu}
	+ 4 \sum\limits_{\mu,\kappa=1}^{\dimR} \nabla_{\Ga[X]} \pi_{\alpha\mu} \cdot \nabla_{\Ga[X]} \pi_{\mu\kappa} H_\kappa.
	\end{aligned}
	\end{equation}
\end{lem}
\begin{proof}
	The lemma is proved in the Appendix~\ref{Appendix A}, in order for the paper to avoid local definitions as much as possible.
	The result follows from Lemma~\ref{heat.equation.for.pi.v2} and Lemma~\ref{heat.equation.for.H.v2}, in a differential geometric setting using calculations in geodesic normal coordinates.
\end{proof}

%

These equations are formally the same as the analogous ones for mean curvature flow (the normal vector and mean curvature) see \cite{Huisken1984}, or \cite[equations (2.4) and (2.5)]{MCF}, the non-linear terms are however more complicated.

\subsection{A coupled system for mean curvature flow}

The evolution of a submanifold of dimension $\dimGa$ in codimension $\dimR-\dimGa \geq 1$ evolving by mean curvature flow is then governed by the coupled system \eqref{eq:velocity law}, \eqref{eq:evolution eq - pi}--\eqref{eq:evolution eq - H} together with the ODE \eqref{eq:velocity ODE}. 
The numerical method is based on the weak form of the above coupled system which reads:
\begin{equation}
\label{evolutioneqs-weak}
\begin{aligned}
v &= \vecH , \\
\int_{\Ga[X]} \!\!\! \mat \pi \cdot \vphi^\pi + \int_{\Ga[X]} \!\!\! \nb_{\Ga[X]} \pi \cdot \nb_{\Ga[X]} \vphi^\pi &=  \int_{\Ga[X]} \!\!\! f_1(\pi) \cdot \vphi^\pi , \\
\int_{\Ga[X]} \!\!\! \mat \vecH \cdot \vphi^{\vecH} + \int_{\Ga[X]} \!\!\! \nb_{\Ga[X]} \vecH \cdot \nb_{\Ga[X]} \vphi^{\vecH} 	&=  \int_{\Ga[X]} \!\!\! f_2(\pi,\vecH) \cdot \vphi^{\vecH}, \\[2mm]
\text{together with the ODE} \qquad 
\pa_t X	&= v \circ X ,
\end{aligned}
\end{equation}
for all test functions $\vphi^\pi \in H^1(\Ga[X])^{\dimR \times \dimR}$ and $\vphi^{\vecH} \in H^1(\Ga[X])^\dimR$.  
This system is complemented with the initial data for $X^0$, $\pi^0$ and $\vecH^0$.

For simplicity, by $\, \cdot \,$ we denote both the Euclidean scalar product for vectors, and the Frobenius inner product for matrices (i.e., the Euclidean product with an arbitrary vectorisation).

We directly compare now the weak formulation of the coupled geometric system for mean curvature flow in codimension $1$, derived in \cite{MCF}, see equation~(2.6): Find the velocity $v$, scalar mean curvature $H$, outward unit normal vector $\nu$, and the parametrisation $X$ such that the following system holds
\begin{equation}
\label{eq:MCF weak form}
\begin{aligned}
&	 \int_{\Ga[X]} \!\!\! \nabla_{\Ga[X]} v \cdot  \nabla_{\Ga[X]} \phiv +  \int_{\Ga[X]} \!\!\! v \cdot \phiv 
\\ 
& \hskip 2cm
=   -\!\int_{\Ga[X]} \!\!\! \nabla_{\Ga[X]}(H\n) \cdot \nabla_{\Ga[X]}\phiv  -\!\int_{\Ga[X]} \!\!\! H\n \cdot \phiv , \\
&	 \int_{\Ga[X]} \!\!\! \mat \n \cdot \phin + \int_{\Ga[X]} \!\!\! \nabla_{\Ga[X]} \n \cdot \nabla_{\Ga[X]} \phin =  \int_{\Ga[X]} \!\!\! | \nabla_{\Ga[X]} \n|^2\,   \n\, \cdot \phin ,
\\
&	 \int_{\Ga[X]} \!\!\! \mat H \, \phiH + \int_{\Ga[X]} \!\!\! \nabla_{\Ga[X]} H \cdot  \nabla_{\Ga[X]} \phiH =   \int_{\Ga[X]} \!\!\! | \nabla_{\Ga[X]} \n|^2 \,  H\, \phiH , \\[2mm]
&\ \text{together with the ODE} \qquad 
\pa_t X	= v \circ X ,
\end{aligned}
\end{equation}
for all test functions $\phiv \in H^1(\Ga[X])^3$ and $\phin \in H^1(\Ga[X])^3$, $\phiH \in H^1(\Ga[X])$. This system is complemented with the initial data $X^0$, $\n^0$ and $H^0$. 

It is also worthwhile to compare the size of the two formulations \eqref{evolutioneqs-weak} and \eqref{eq:MCF weak form} for a surface of codimension $1$ in $\R^n$: Without the ODE present in both cases, the weak formulation \eqref{eq:MCF weak form} is of size $2n + 1$, while the new weak system \eqref{evolutioneqs-weak} is of size $n^2 + n$ (the first equation is merely an identity).

We note that the first equation determining $v$ could be simplified to the natural pointwise identity $v=- H \nu$, see \cite{Willmore}. 

\section{Evolving surface finite element semi-discretization}
\label{section:ESFEM}


\subsection{Evolving surface finite elements}
We formulate the evolving surface finite element (ESFEM) discretization for the velocity law coupled with evolution equations on the evolving surface, following the description in \cite{KLLP2017,MCF}, which is based on \cite{Dziuk88,Demlow2009,highorderESFEM}. We use simplicial finite elements and continuous piecewise polynomial basis functions of degree~$k$, as defined in \cite[Section 2.5]{Demlow2009}.

We triangulate the given smooth initial surface $\Ga^0$ by an admissible family of triangulations $\mathcal{T}_h$ of decreasing maximal element diameter $h$; see \cite{DziukElliott_ESFEM} for the notion of an admissible triangulation, which includes quasi-uniformity and shape regularity. For a momentarily fixed $h$, we denote by $\bfx^0 $  the vector in $\R^{\dimR \dof}$ that collects all nodes $p_j$ $(j=1,\dots,\dof)$ of the initial triangulation. By piecewise polynomial interpolation of degree $k$, the nodal vector defines an approximate surface $\Ga_h^0$ that interpolates $\Ga^0$ in the nodes $p_j$. We will evolve the $j$th node in time, denoted $x_j(t)$ with $x_j(0)=p_j$, and collect the nodes at time $t$ in a column vector
$$
\bfx(t) \in \R^{\dof \dimR}. 
$$
We just write $\bfx$ for $\bfx(t)$ when the dependence on $t$ is not important.

By piecewise polynomial interpolation on the  plane reference triangle that corresponds to every
curved triangle of the triangulation, the nodal vector $\bfx$ defines a closed surface denoted by $\Ga_h[\bfx]$. We can then define globally continuous finite element {\it basis functions}
$$
\phi_i[\bfx] \colon \Ga_h[\bfx]\to\R, \qquad i=1,\dotsc,\dof,
$$
which have the property that on every triangle their pullback to the reference triangle is polynomial of degree $k$, and which satisfy at the nodes $\phi_i[\bfx](x_j) = \delta_{ij}$ for all $i,j = 1,  \dotsc, \dof .$
These functions span the finite element space on $\Ga_h[\bfx]$,
\begin{equation*}
S_h[\bfx] = S_h(\Ga_h[\bfx])=\spn\big\{ \phi_1[\bfx], \phi_2[\bfx], \dotsc, \phi_\dof[\bfx] \big\} .
\end{equation*}
For a finite element function $u_h\in S_h[\bfx]$, the tangential gradient $\nabla_{\Ga_h[\bfx]}u_h$ is defined piecewise on each element.

The discrete surface at time $t$ is parametrized by the initial discrete surface via the map $X_h(\cdot,t) \colon \Ga_h^0\to\Ga_h[\bfx(t)]$ defined by
$$
X_h(p_h,t) = \sum_{j=1}^\dof x_j(t) \, \phi_j[\bfx(0)](p_h), \qquad p_h \in \Ga_h^0,
$$
which has the properties that $X_h(p_j,t)=x_j(t)$ for $j=1,\dotsc,\dof$, that  $X_h(p_h,0) = p_h$ for all $p_h\in\Ga_h^0$, and
$$
\Ga_h[\bfx(t)]=\Ga[X_h(\cdot,t)] = \{ X_h(p_h,t) \,:\, p_h \in \Ga_h^0 \}. 
$$

The {\it discrete velocity} $v_h(x,t) \in \R^\dimR$ at a point $x=X_h(p_h,t) \in \Ga[X_h(\cdot,t)]$ is given by
$$
\partial_t X_h(p_h,t) = v_h(X_h(p_h,t),t).
$$
In view of the transport property of the basis functions  \cite{DziukElliott_ESFEM},
$
\tfrac\d{\d t} \big( \phi_j[\bfx(t)](X_h(p_h,t)) \big) =0 ,
$
the discrete velocity equals, for $x \in \Ga_h[\bfx(t)]$,
$$
v_h(x,t) = \sum_{j=1}^\dof v_j(t) \, \phi_j[\bfx(t)](x) \qquad \hbox{with } \ v_j(t)=\dot x_j(t),
$$
where the dot denotes the time derivative $\d/\d t$. 
Hence, the discrete velocity $v_h(\cdot,t)$ is in the finite element space $S_h[\bfx(t)]$, with nodal vector $\bfv(t)=\dot\bfx(t)$.
%
%

The {\it discrete material derivative} of a finite element function $u_h(x,t)$ with nodal values $u_j(t)$ is
$$
\mat_h u_h(x,t) = \frac{\d}{\d t} u_h(X_h(p_h,t)) = \sum_{j=1}^\dof \dot u_j(t)  \phi_j[\bfx(t)](x)  \quad\text{at}\quad x=X_h(p_h,t).
$$


\subsection{ESFEM spatial semi-discretizations}
\label{subsection:semi-discretization}

Now we will describe the semi-discretization of the coupled system for mean curvature flow in arbitrary codimension.

The finite element spatial semi-discretization of the weak coupled parabolic system \eqref{evolutioneqs-weak} reads as follows: Find the unknown nodal vector $\bfx(t)\in \R^{\dof \dimR}$ and the unknown finite element functions $v_h(\cdot,t)\in S_h[\bfx(t)]^\dimR$ and $\pi_h(\cdot,t)\in S_h[\bfx(t)]^{\dimR\times \dimR}$, and $\vecH_h(\cdot,t) \in S_h[\bfx(t)]^\dimR$ satisfying the coupled semi-discrete system:
\begin{subequations}
	\label{eq:semi-discretization}
	\begin{align}
	v_h &= \vecH_h , \\
	\int_{\Ga_h[\bfx]} \mat_h \pi_h \cdot \vphi^\pi_h
	+ \int_{\Ga_h[\bfx]} \nb_{\Ga_h[\bfx]} \pi_h \cdot \nb_{\Ga_h[\bfx]} \vphi^\pi_h 
	&=  \int_{\Ga_h[\bfx]} f_1(\pi_h) \cdot \vphi^\pi_h, \\
	\int_{\Ga_h[\bfx]} \mat_h \vecH_h \cdot \vphi^{\vecH}_h 
	+ \int_{\Ga_h[\bfx]} \nb_{\Ga_h[\bfx]} \vecH_h \cdot \nb_{\Ga_h[\bfx]} \vphi^{\vecH}_h 
	&= \int_{\Ga_h[\bfx]} f_2(\pi_h,\vecH_h) \cdot \vphi^{\vecH}_h ,
	\end{align}
\end{subequations}
where $f_1(\pi_h)$ and $f_2(\pi_h,\vecH_h)$ are the spatially discrete analogons of the non-linear expressions \eqref{eq:non-linear expressions - f_1 and f_2}, 
for all $\vphi^\pi_h \in S_h[\bfx(t)]^{\dimR\times \dimR}$ and $\vphi^{\vecH}_h \in S_h[\bfx(t)]^\dimR$, with the surface $\Ga_h[\bfx(t)]=\Ga[X_h(\cdot,t)] $ given by the differential equation
\begin{equation}
\label{eq:xh}
\partial_t X_h(p_h,t) = v_h(X_h(p_h,t),t), \qquad p_h\in\Ga_h^0.
\end{equation}
The initial values for the nodal vector $\bfx$ are taken as the positions of the nodes of the triangulation of the given initial surface $\Ga^0$.
The initial data $\pi_h^0$ and $\vecH_h^0$ are determined by componentwise Lagrange interpolation of $\pi^0$ and $\vecH^0$. 

\subsection{Matrix--vector formulation}
\label{section:matrix-vector form}

The nodal values of the unknown semi-discrete functions $v_h(\cdot,t) \in S_h[\bfx(t)]^{\dimR}$, $\pi_h(\cdot,t) \in S_h[\bfx(t)]^{\dimR\times \dimR}$, and $\vecH(\cdot,t) \in S_h[\bfx(t)]^{\dimR}$ are collected into column vectors  $\bfv\t = (v_j\t) \in \R^{\dof  \dimR}$, $\bfpi\t = (\pi_j\t) \in \R^{\dof  \dimR^2}$, and $\bfvecH\t = (\vecH_j\t)\in\R^{\dof  \dimR}$, respectively. We furthermore set
\begin{equation*}
\bfu = \begin{pmatrix} \bfpi \\ \bfvecH \end{pmatrix} \in \R^{\dof (\dimR^2+\dimR)} ,
\end{equation*}
and set $\bfI$ to be a block matrix extracting the $\bfvecH$ component of $\bfu$, that is $\bfI \bfu = \bfvecH$.

We define the surface-dependent mass matrix $\bfM(\bfx)$ and stiffness matrix $\bfA(\bfx)$:
\begin{align*}
\bfM(\bfx)|_{ij} = \int_{\Ga_h[\bfx]} \phi_i[\bfx] \phi_j[\bfx] \qand \bfA(\bfx)|_{ij} = \int_{\Ga_h[\bfx]} \nb_{\Ga_h[\bfx]} \phi_i[\bfx] \cdot \nb_{\Ga_h[\bfx]} \phi_j[\bfx] , 
\end{align*}
for $i,j = 1,  \dotsc,\dof$.
The non-linear terms $\bff(\bfx,\bfu) = \big(\bff_1(\bfx,\bfu) , \bff_2(\bfx,\bfu)\big)^T$ are defined by
\begin{align*}
\bff_1(\bfx,\bfu)|_{k+(\alpha-1)\dof+(\beta-1) \dimR \dof} 
= &\ \int_{\Ga_h[\bfx]} f_1(\pi_h)_{\alpha\beta} \ \phi_k[\bfx] , \\
\bff_2(\bfx,\bfu)|_{k+(\alpha-1)\dof} 
= &\ \int_{\Ga_h[\bfx]} f_2(\pi_h,\vecH_h)_{\alpha} \ \phi_k[\bfx] , \\
\end{align*}
for $k = 1, \dotsc, \dof$ and $\alpha,\beta=1,\dotsc,\dimR$. 

We further let, for $d \in \N$ (with the identity matrices $I_d \in \R^{d \times d}$) 
$$
\bfM^{[d]}(\bfx)= I_d \otimes \bfM(\bfx), \qquad
\bfA^{[d]}(\bfx)= I_d \otimes \bfA(\bfx) .
$$
When no confusion can arise, we will write $\bfM(\bfx)$ for $\bfM^{[d]}(\bfx)$, and $\bfA(\bfx)$ for $\bfA^{[d]}(\bfx)$.

Using these definitions \eqref{eq:semi-discretization} with \eqref{eq:xh} can be written in the matrix--vector form:
\begin{equation}
\label{eq:matrix-vector form}
\begin{aligned}
\bfv &= \bfI \bfu , \\
\bfM(\bfx) \dot{\bfu} + \bfA(\bfx) \bfu &= \bff(\bfx,\bfu), \\
\text{with} \qquad 
\dot \bfx &= \bfv .
\end{aligned} 
\end{equation}

The above matrix--vector formulation \eqref{eq:matrix-vector form} for mean curvature flow in arbitrary codimension is almost identical to the same formulas for mean curvature flow in codimension $1$ \cite[equation~(3.4)--(3.5)]{MCF}:
\begin{align*}
\big( \bfM(\bfx) + \bfA(\bfx) \big) \bfv &= \bfg(\bfx,\bfu), \\
\bfM(\bfx) \dot{\bfu} + \bfA(\bfx) \bfu &= \bff(\bfx,\bfu), \\
\text{with} \qquad 
\dot \bfx &= \bfv .
\end{align*} 
In the two above ODE systems the equations for $\bfu$ and $\bfx$ are formally the same, the equation for $\bfv$ is even simpler here.
Note that here $\bfu$ collects $\bfpi$ and $\bfvecH$, whereas for mean curvature flow in codimension $1$ it collects $\bfu = (\bfn,\bfH)^T$, the nodal values of the approximations to the normal vector and scalar mean curvature. Naturally, the block-size of the matrices in the two equations for $\bfu$ are greatly different (respectively, $\dimR^2 + \dimR$ and $\dimR + 1$).

It is crucial to notice that, thanks to the coinciding matrix--vector formulations many results from \cite{MCF}, most notably the stability results Proposition~7.1 and Proposition~10.1 therein, hold directly for the present case as well.

\subsection{Lifts}
\label{section:lifts}

As in \cite{KLLP2017} and \cite[Section~3.4]{MCF}, we compare functions on the {\it exact surface} $\Ga[X(\cdot,t)]$ with functions on the {\it discrete surface} $\Ga_h[\bfx(t)]$, via functions on the {\it interpolated surface} $\Ga_h[\xs(t)]$, where
$\xs(t)$ denotes the nodal vector collecting the grid points $x_j^*(t)=X(p_j,t)$ on the exact surface, where $p_j$ are the nodes of the discrete initial triangulation $\Ga_h^0$.

Any finite element function $w_h$ on the discrete surface, with nodal values $w_j$, is associated with a finite element function $\widehat w_h$ on the interpolated surface $\Ga_h[\xs]$ with the exact same nodal values. 
This can be further lifted to a function on the exact surface by using the \emph{lift operator} $\,^\ell$, mapping a function on the interpolated surface $\Ga_h[\xs]$ to a function on the exact surface $\Ga[X]$,
via the \emph{closest point projection}. Provided that the two surfaces are sufficiently close, for $x \in \Ga_h[\xs]$ find $x^\ell \in \Ga[X]$ such that $x^\ell - x$ is minimal, i.e.
\begin{equation*}
x^\ell - x \perp T_{x^\ell}\Ga[X] , \qquad \text{and then setting} \qquad \widehat w_h^\ell(x^\ell) = \widehat w_h(x) .
\end{equation*}
This definition is consistent with the lift operator in codimension $1$, see \cite{Dziuk88,DziukElliott_ESFEM,Demlow2009}, using the signed distance function $d$.
The standard norm-equivalence results (\cite[equation (2.15)--(2.17)]{Demlow2009}) hold for this definition as well.

Then the composed lift $\,^L$ maps finite element functions on the discrete surface $\Ga_h[\bfx]$ to functions on the exact surface $\Ga[X]$ via the interpolated surface $\Ga_h[\xs]$, and it is defined by 
$$
w_h^L = (\widehat w_h)^\ell.
$$

\section{Linearly implicit full discretization}
\label{section:BDF}

Similarly as for mean curvature flow \cite{MCF}, for the time discretization of the system of ordinary differential equations \eqref{eq:matrix-vector form} we use a $q$-step linearly implicit backward difference formula (BDF method). For a step size $\tau>0$, and with $t_n = n \tau \leq T$, we determine the approximations to all variables $\bfx^n$ to $\bfx(t_n)$, $\bfv^n$ to $\bfv(t_n)$, and $\bfu^n$ to $\bfu(t_n)$ by the fully discrete system of \emph{linear} equations
\begin{subequations}
	\label{eq:BDF}
	\begin{align}
	\bfv^n &= \bfI \, \bfu^n , \\
	\bfM(\widetilde \bfx^n) \dot \bfu^n + \bfA(\widetilde \bfx^n) \bfu^n &= \bff(\widetilde \bfx^n,\widetilde \bfu^n) ,  \\
	\dot \bfx^n &=  \bfv^n,
	\end{align}
\end{subequations}
where we denote the discretized time derivatives
\begin{equation}
\label{eq:backward differences def}
\dot \bfx^n = \frac{1}{\tau} \sum_{j=0}^q \delta_j \bfx^{n-j} , \qquad	\dot \bfu^n = \frac{1}{\tau} \sum_{j=0}^q \delta_j \bfu^{n-j} , \qquad n \geq q ,
\end{equation}
and where $\widetilde \bfx^n$ and $\widetilde \bfu^n$ are the extrapolated values 
\begin{equation}
\label{eq:extrapolation def}
\widetilde \bfx^n = \sum_{j=0}^{q-1} \gamma_j \bfx^{n-1-j} , \qquad	\widetilde \bfu^n = \sum_{j=0}^{q-1} \gamma_j \bfu^{n-1-j} , \qquad n \geq q .
\end{equation}
The starting values $\bfx^i$ and $\bfu^i$ ($i=0,\dotsc,q-1$) are assumed to be given; in addition we set $\widetilde \bfx^i = \bfx^i$ and $\widetilde \bfu^i = \bfu^i$ for $i=0,\dotsc,q-1$. They can be precomputed using either a lower order method with smaller step sizes or an implicit Runge--Kutta method.

The method is determined by its coefficients, given by $\delta(\zeta)=\sum_{j=0}^q \delta_j \zeta^j=\sum_{\ell=1}^q \frac{1}{\ell}(1-\zeta)^\ell$ and $\gamma(\zeta) = \sum_{j=0}^{q-1} \gamma_j \zeta^j = (1 - (1-\zeta)^q)/\zeta$. 
The classical BDF method is known to be zero-stable for $q\leq6$ and to have order $q$; see \cite[Chapter~V]{HairerWannerII}.
This order is retained, for $q \leq 5$, by the linearly implicit variant using the above coefficients $\gamma_j$; cf.~\cite{LubichMansourVenkataraman_bdsurf,AkrivisLiLubich_quasilinBDF}.

We again point out that the fully discrete system \eqref{eq:BDF}--\eqref{eq:approx} is formally the same as the fully discrete system for the mean curvature flow for surfaces \cite[equations (5.1)--(5.4)]{MCF}. 
Theorem~6.1 in \cite{MCF} proves optimal-order error bounds for the combined ESFEM--BDF full discretization
of the mean curvature flow system, for finite elements of polynomial degree $k \geq 2$ and BDF methods of order $2 \leq q \leq 5$.

We note that in the $n$th time step, the method decouples and hence only requires solving a linear system with the symmetric positive definite matrix $\delta_0 \bfM(\widetilde \bfx^n) + \tau\bfA(\widetilde \bfx^n)$.

From the vectors and matrices $\bfx^n =(x_j^n)$, $\bfv^n = (v_j^n)$, and $\bfu^n=(u_j^n)$ with $u_j^n=(\pi_j^n,\vecH_j^n)$, where $\pi_j^n \in\R^{\dimR\times \dimR}$ and $\vecH_j^n \in \R^\dimR$, we obtain position approximations to $X(\cdot,t_n)$,  $\Id_{\Ga[X(\cdot,t_n)]}$, velocity approximations to $v(\cdot,t_n)$, and approximations to the orthogonal projection and the mean curvature vector, respectively, at time $t_n$ as
\begin{equation}
\label{eq:approx}
\begin{aligned}
X_h^n(p_h) &= \sum_{j=1}^\dof x_j^n \, \phi_j[\bfx(0)](p_h) \quad\hbox{ for } p_h \in \Ga_h^0, \\   
x_h^n(x) & = \Id_{\Ga[X_h^n]}, \\
v_h^n(x)      &= \sum_{j=1}^\dof v_j^n \, \phi_j[\bfx^n](x)  \qquad\hbox{ for } x \in \Ga_h[\bfx^n], \\
\pi_h^n(x)      &= \sum_{j=1}^\dof \pi_j^n \, \phi_j[\bfx^n](x)  \qquad\hbox{ for } x \in \Ga_h[\bfx^n], \\
\vecH_h^n(x)      &= \sum_{j=1}^\dof \vecH_j^n \, \phi_j[\bfx^n](x)  \qquad\hbox{ for } x \in \Ga_h[\bfx^n].
\end{aligned}
\end{equation}
In the semi-discrete case, the approximations of the same quantities are given analogously.

\section{Main results: error estimates}
\label{section:main results}

We are now in the position to state the main results of this paper, time uniform optimal-order semi- and fully discrete $H^1$-norm error estimates for the position, velocity, orthogonal projection, and mean curvature vector obtained, respectively, by the semi-discretization \eqref{eq:semi-discretization} (or \eqref{eq:matrix-vector form}), or the linearly implicit BDF full discretization \eqref{eq:BDF}, using evolving surface finite elements of polynomial degree at least $2$, and $q$-step BDF method with $2 \leq q \leq 5$.

\subsection{Convergence of the semi-discretization}
\begin{thm}
	\label{theorem:semi-discrete convergence}
	Consider the semi-discretization \eqref{eq:semi-discretization} of the mean curvature flow \eqref{eq:velocity law} 
	in arbitrary codimension $\dimR-\dimGa$, using evolving surface finite elements of polynomial degree $k \geq 2$.
	Suppose that the mean curvature flow problem in arbitrary codimension admits an exact solution $(X,v,\pi,\vecH)$ that is sufficiently smooth on the time interval $t \in [0,T]$, and that the flow map $X(\cdot,t) \colon \Ga^0 \to \Ga(t) \subset \R^\dimR$ is non-degenerate so that $\Ga(t) = \Ga[X(\cdot,t)]$ is a regular surface on the time interval $t \in [0,T]$. 
	
	Then there exist constants $h_0 > 0$ and $C > 0$ such that 
	\begin{alignat*}{3}
	&\ \| x_h^L (\cdot,t) - \Id_{\Ga(t)} \|_{H^1(\Ga(t))} \leq C h^k, \qquad
	& &\ \| v_h^L (\cdot,t) - v(\cdot,t) \|_{H^1(\Ga(t))} \leq C h^k, \\
	&\ \| \pi_h^L (\cdot,t) - \pi(\cdot,t) \|_{H^1(\Ga(t))} \leq C h^k, \qquad
	& &\ \| \vecH_h^L (\cdot,t) - \vecH(\cdot,t) \|_{H^1(\Ga(t))} \leq C h^k , \\
	& \qquad \qquad\qquad\qquad\text{and}  & &\ \| X_h^\ell (\cdot,t) - X(\cdot,t) \|_{H^1(\Ga^0)} \leq C h^k ,
	\end{alignat*}
	for all $h \leq h_0$.
	The constant $C>0$ is independent of $h$ and $t$, but depends on the $H^{k+1}$-norms of the exact solution $(X,v,\pi,\vecH)$ and on the final time $T$. 
\end{thm}

\begin{proof}
	The result essentially follows from the proof of Theorem~4.1 in \cite{MCF}. 
	
	The stability is shown following the proof of Proposition~7.1 in \cite{MCF}, since (as we have pointed out above) the matrix--vector formulation \eqref{eq:matrix-vector form} is (almost) identical to the matrix--vector formulation of \cite[equation~(3.4)--(3.5)]{MCF} (recalling that here $\bfu =(\bfpi,\bfvecH)^T$ is in the role of $\bfu=(\bfn,\bfH)^T$ in \cite{MCF}). The system uses the same mass and stiffness matrices (but of different size), while the proof therein only uses the local Lipschitz continuity of the non-linear terms, which holds here as well. The bounded operator $\bfI$ in the velocity equation $\bfv = \bfI \bfu$ even simplifies part (B) of the stability proof of Proposition~7.1 in \cite{MCF}.
	
	The consistency errors for $(X,v,\pi,\vecH)$ are shown by the exact techniques of the consistency analysis \cite[Lemma~8.1]{MCF}.
	
	The uniform-in-time $H^1$-norm error bounds are proved by combining stability and consistency, verbatim as in \cite[Section~9]{MCF}. 
\end{proof}

\subsection{Convergence of the full discretization}

\begin{thm}
	\label{theorem:fully discrete convergence}
	Consider the full discretization \eqref{eq:BDF} of the mean curvature flow \eqref{eq:velocity law} in arbitrary codimension $\dimR-\dimGa$, using evolving surface finite elements of polynomial degree $k \geq 2$ and linearly implicit BDF time discretization of order $q$ with $2 \leq q \leq 5$. 
	Suppose that the mean curvature flow problem in arbitrary codimension admits an exact solution $(X,v,\pi, \vecH)$ that is sufficiently smooth on the time interval $t \in [0,T]$, and that the flow map $X(\cdot,t) \colon \Ga^0 \to \Ga(t) \subset \R^\dimR$ is non-degenerate so that $\Ga(t) = \Ga[X(\cdot,t)]$ is a regular surface on the time interval $t \in [0,T]$.
	
	Then, there exist constants $h_0 > 0$, $\tau_0 > 0$ and $C_0 > 0$ such that for all mesh sizes
	$h \leq h_0$ and time step sizes $\tau \leq \tau_0$ satisfying the mild step size restriction
	\begin{equation}
	\tau \leq C_0 h 
	\end{equation}
	(where $C_0 > 0$ can be chosen arbitrarily), the following error bounds for the
	lifts of the discrete position, velocity, tangential projection and mean curvature vector hold
	over the exact surface: provided that the starting values are $\mathcal{O}(h^k + \tau^{q+1/2})$)
	accurate in the $H^1$ norm at time $t_i$ for  $i=0,\dotsc,q-1$, we have at time
	$t_n = n\tau \leq T$
	\begin{align*}
	\| (x_h^n)^L - \Id_{\Ga(t_n)} \|_{H^1(\Ga(t_n))} \leq &\ C (h^k+\tau^q), \\
	\| (v_h^n)^L - v(\cdot,t_n) \|_{H^1(\Ga(t_n))} \leq &\ C (h^k+\tau^q), \\
	\| (\pi_h^n)^L - \pi(\cdot,t_n) \|_{H^1(\Ga(t_n))} \leq &\ C (h^k+\tau^q), \\
	\| (\vecH_h^n)^L - \vecH(\cdot,t_n) \|_{H^1(\Ga(t_n))} \leq &\ C (h^k+\tau^q) , \\
	\text{and} \qquad \| (X_h^n)^\ell - X(\cdot,t_n) \|_{H^1(\Ga^0)} \leq &\ C (h^k+\tau^q) .
	\end{align*}
	for all $h \leq h_0$.
	The constant $C>0$ is independent of $h$, $\tau$ and $n$, but depends on bounds of higher-derivatives of the exact solution $(X,v,\pi,\vecH)$, on the final time $T$ and on $C_0$.
\end{thm}

\begin{proof}
	Similarly to the semi-discrete error bounds: Since the ESFEM / linearly implicit BDF discretization \eqref{eq:BDF} is (almost) identical to equation~(5.1) in \cite{MCF}, the proof of this result directly follows as the proof of Theorem~6.1 in \cite{MCF}.
\end{proof}

\begin{rem}
	\label{remark:dimension limit}
	The stability and convergence results readily extend to higher dimensional submanifolds of dimension $\dimGa \geq 4$ (of arbitrary codimension $\dimR - \dimGa$), cf.~Section~14 of \cite{MCF}, provided that optimal-order quasi-interpolation is used instead of the nodal interpolation (cf.~\cite[Lemma~4.3]{DziukElliott_acta}), and requires evolving surface finite elements of degree $k \geq \lfloor\dimGa / 2 \rfloor + 1$ and BDF methods of order $\lfloor\dimGa / 2 \rfloor + 1 \leq q \leq 5$.
	For the six-step BDF method a new multiplier-based energy technique was developed in \cite{Akrivisetal_BDF6}. The fully discrete stability proof in \cite{MCF} should generalise to this approach.
\end{rem}

\section{Numerical examples for curves
	in $\R^3$}
\label{section:numerics}

We performed the following numerical experiments for mean curvature flow of curves
in $\R^3$:
\begin{itemize}
	\item[-] A convergence test using planar curves where the exact solution is known.
	\item[-] A comparison test with Dziuk's algorithm for curves \cite{DeckelnickDziukElliott_acta} using circles and Angenent ovals \cite{Angenent_ovals}.
	\item[-] Experiments for space curves using established examples from the literature \cite{BGN_curve_grad_flow,Pozzi_CSF_codim}, e.g.~a trefoil knot, still comparing with Dziuk's algorithm.
\end{itemize}

All our numerical experiments were carried out in Matlab, using quadratic evolving surface finite elements, and BDF methods of various order specified in the experiments. 
The parametrisation of the quadratic elements was inspired by \cite{BCH2006}. 
The initial meshes were all generated using an arc-length parametrisation, without taking advantage of any symmetry of the surface.

\subsection{Convergence test}

We are reporting on the errors of our algorithm for mean curvature flow in codimension $2$ for \emph{flat space curves}. 
Simple test examples are constructed in this setting, by using the fact that the evolution of flat space curves evolving under the flow \eqref{eq:velocity law} is equivalent to their evolution under \emph{curve shortening flow}.

Let the curve $\Ga^0:[0,2\pi] \to \R^3$ be a circle of initial radius $R_0$ in an arbitrary plane.

We consider the mean curvature flow of $\Ga(\cdot,t)$ with initial value $\Ga_0$. 
Using the rotational symmetry of $\Ga$ along flow, we obtain that its radius satisfies the ODE:
\begin{equation}
\label{eq:CSF - radius ODE}
\diff R(t) = - \frac{1}{R(t)} , \qquad \text{with initial value} \quad R(0) = R_0 .
\end{equation}
The above initial value problems has the solution $R(t) = \sqrt{R_0^2 - 2 t}$ until final time  $T_{\max} = R_0^2 / 2$. Therefore, the curvature of $\Ga(\cdot,t)$ is given by $H(\cdot,t) = 1 / R(t) = (R_0^2 - 2 t)^{-1/2}$.

We computed numerical approximations to the flow using quadratic finite elements ($k=2$) and using the $2$-step linearly implicit BDF method ($q=2$) for a circle of radius $R_0=1$ which lies in the $y$-$z$-plane rotated by $\theta = \pi/e$. The starting values $\bfx^i \in \R^{3N}$ and $\bfu = (\bfpi , \bfH)^T \in \R^{(9+3)N}$ for $i=1,\dotsc,q-1$ were computed as the interpolations of the exact values.

In Figure~\ref{fig:circle space conv} and \ref{fig:circle time conv} we report on the errors between the numerical and (interpolation of) exact solutions for mean curvature flow in codimension $2$ of a flat circle until the final time $T_{\max}$, illustrating the error bounds of Theorem~\ref{theorem:semi-discrete convergence} and \ref{theorem:fully discrete convergence}.
The two plots in Figure~\ref{fig:circle space conv} report on the surface error and the errors of the mean curvature $\vecH_h$, respectively on the left- and right-hand side. 
The logarithmic error plots show the $L^\infty(H^1)$ norm errors against the mesh size $h$. The lines marked with different symbols correspond to different time step sizes $\tau$.
Figure~\ref{fig:circle time conv} reports on the same errors but reversing roles of $h$ and $\tau$.

In both cases the error curves match the slope of the reference lines (dashed) corresponding to the convergence order of Theorem~\ref{theorem:semi-discrete convergence} and \ref{theorem:fully discrete convergence}, $\mathcal{O}(h^2)$ and $\mathcal{O}(\tau^2)$

\begin{figure}[htbp]
	\includegraphics[width=\textwidth]{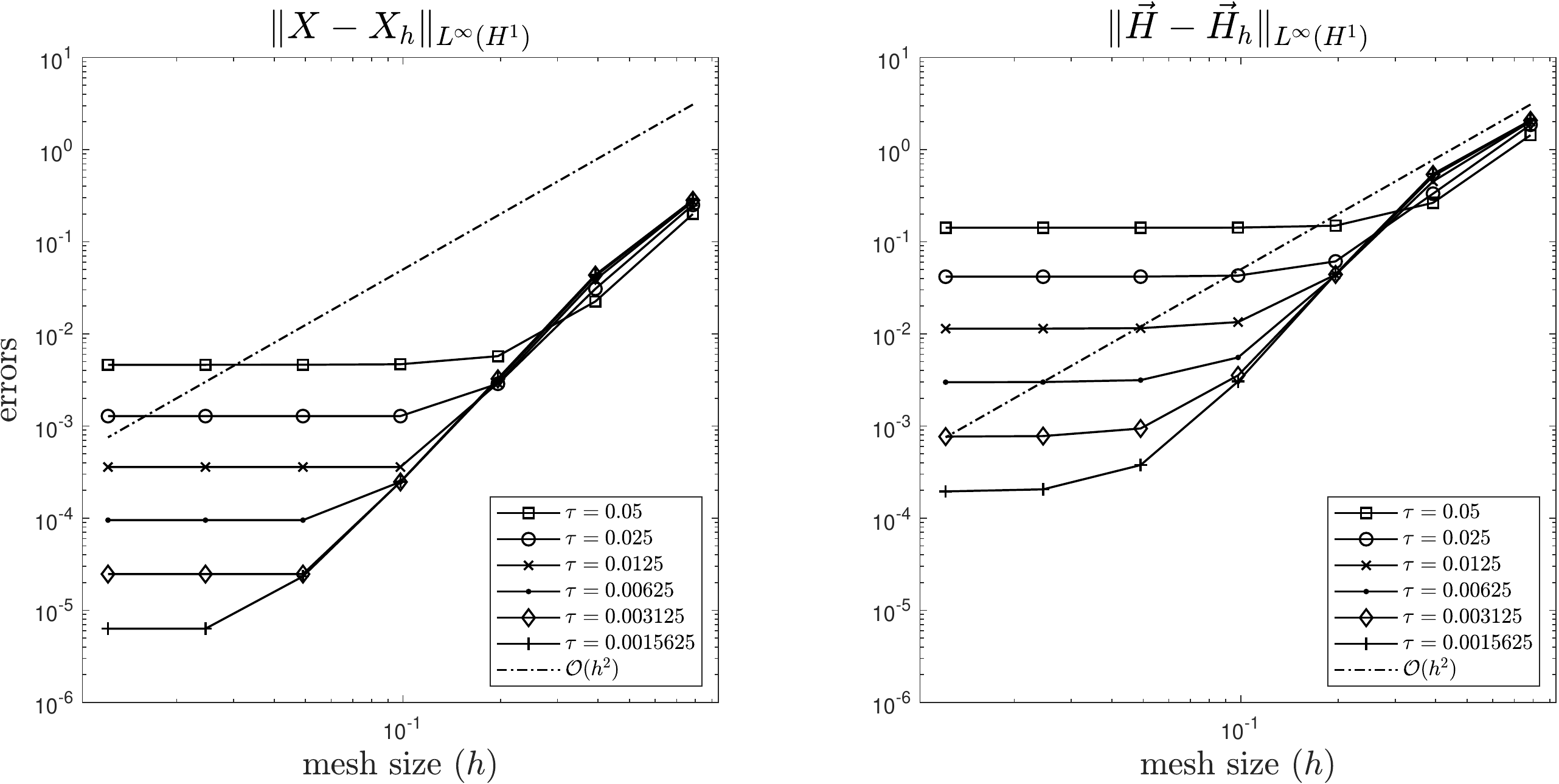}
	\caption{Spatial convergence of the BDF2 / quadratic ESFEM discretization for MCF codimension 2 of the unit circle for $T = 0.4$.}
	\label{fig:circle space conv}
\end{figure}

\begin{figure}[htbp]
	\includegraphics[width=\textwidth]{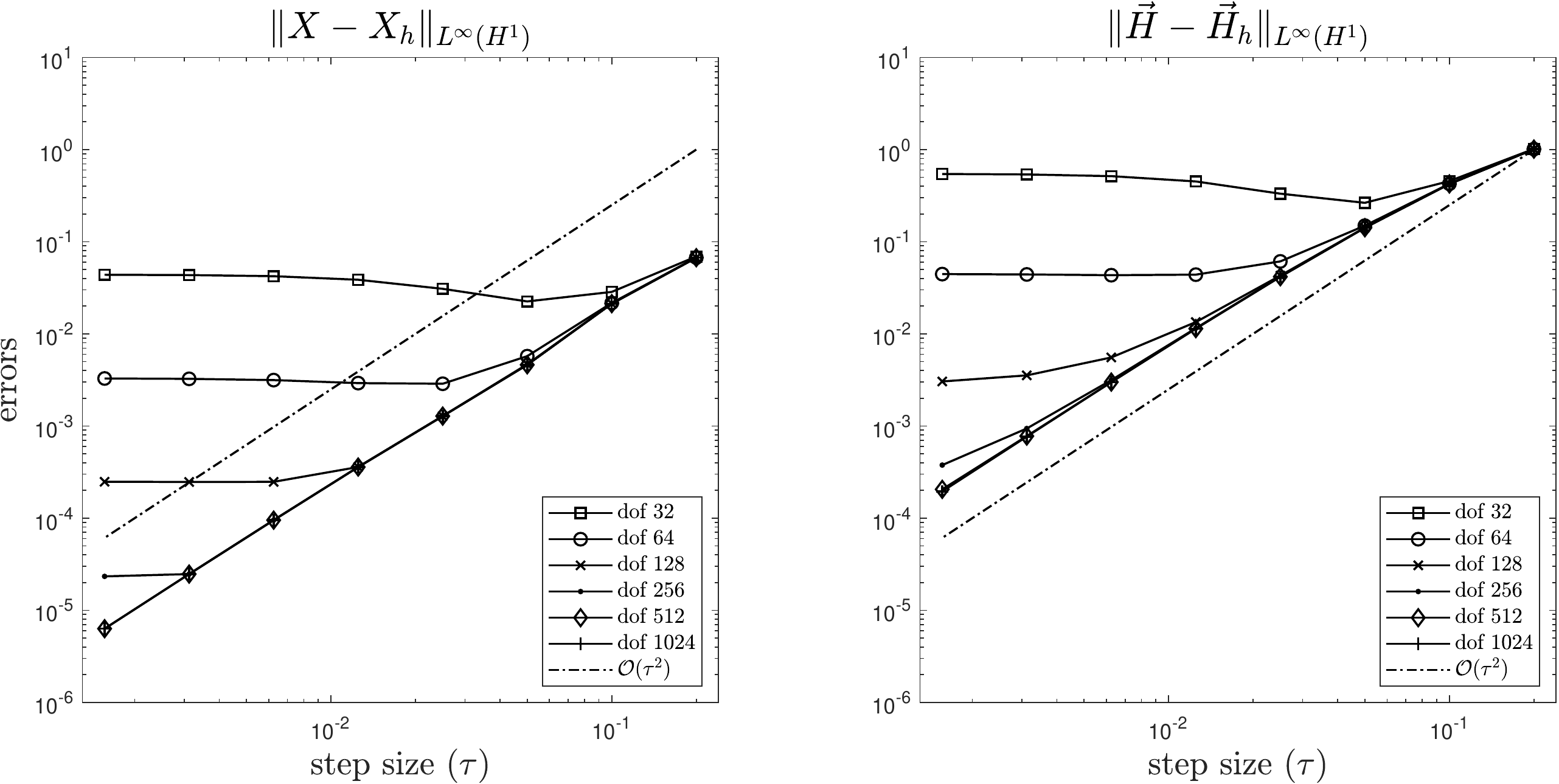}
	\caption{Temporal convergence of the BDF2 / quadratic ESFEM discretization for MCF codimension 2 of the unit circle for $T = 0.4$.}
	\label{fig:circle time conv}
\end{figure}

\subsection{Comparison with Dziuk's algorithm}

We compared the algorithm \eqref{eq:BDF} with (the linearly implicit BDF version of) Dziuk's algorithm for curves, see~\cite{Dziuk_CSF_1994,DeckelnickDziukElliott_acta}:
\begin{equation}
\label{eq:Dziuk's alg - matrix-vector form}
\bfM(\widetilde \bfx^n) \dot \bfx^n + \bfA(\widetilde \bfx^n) \bfx^n = 0 , \qfor n \geq q ,
\end{equation}
with given initial data $\bfx^i \in \R^{3N}$ for $i=1,\dotsc,q-1$.

Figure~\ref{fig:compare circle} compares the exact solution (black), Dziuk's algorithm (grey), and our algorithm \eqref{eq:BDF} (light grey) for a flat circle of unit radius over the time interval $[0,0.4875]$, using a mesh with $128$ nodes and $\tau = 0.0125$.
\begin{figure}[htbp]
	\includegraphics[width=\textwidth]{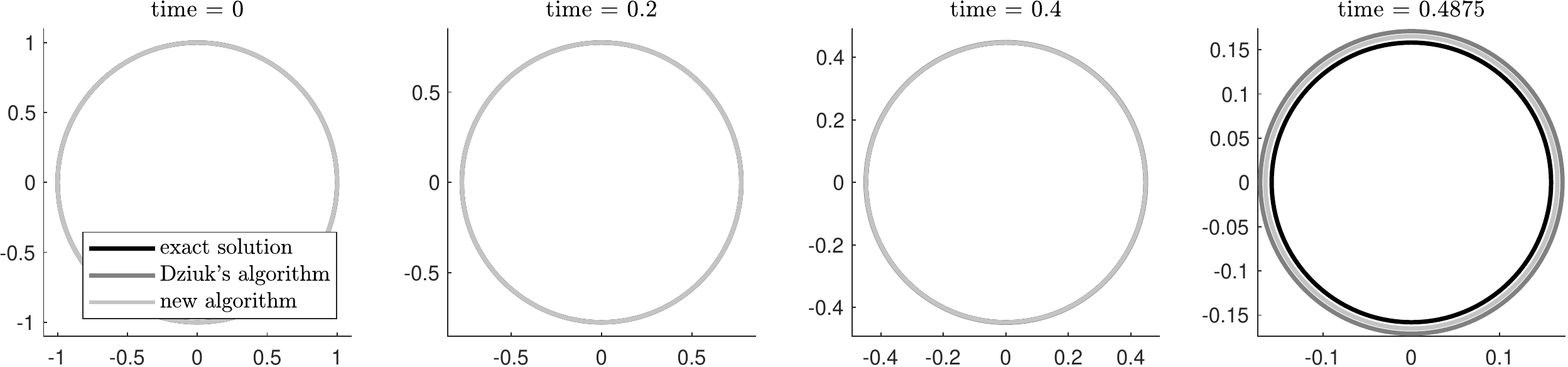}
	\caption{Comparing our algorithm (light grey) with the exact solution (black) and Dziuk's algorithm (grey) using a flat circle.}
	\label{fig:compare circle}
\end{figure}

Figure~\ref{fig:compare Angenent ovals} reports on the same comparison for Angenent ovals, defined, for $\theta \in [0,2\pi]$ and $t \in (-\infty,0)$, by
\begin{equation}
\label{eq:Angenent ovals}
\begin{aligned}
X(\theta,t) = &\ \Big( \int_0^r \cos(\varphi) \kappa(\varphi)^{-1} \d \varphi , \ \int_0^r \cos(\varphi) \kappa(\varphi)^{-1} \d \varphi , \ 0 \Big) , \\
\text{with} \qquad 
\kappa^2(\varphi,t) = &\ (e^{-2t}-1)^{-1} + \cos^2(\varphi) ,
\end{aligned}
\end{equation}
for more details we refer to \cite{Angenent_ovals}. Choosing $\Ga^0$ as the Angenent oval with $t_0 < 0$ via \eqref{eq:Angenent ovals}, a solution exists on the interval $[0,-t_0)$.

The experiment of Figure~\ref{fig:compare Angenent ovals} was performed on the time interval $[0,2]$ using the Angenent oval with $t_0 = -2$ as initial values $\Ga^0$, using a mesh with $128$ nodes and $\tau = 10^{-4}$.
\begin{figure}[htbp]
	\includegraphics[width=\textwidth]{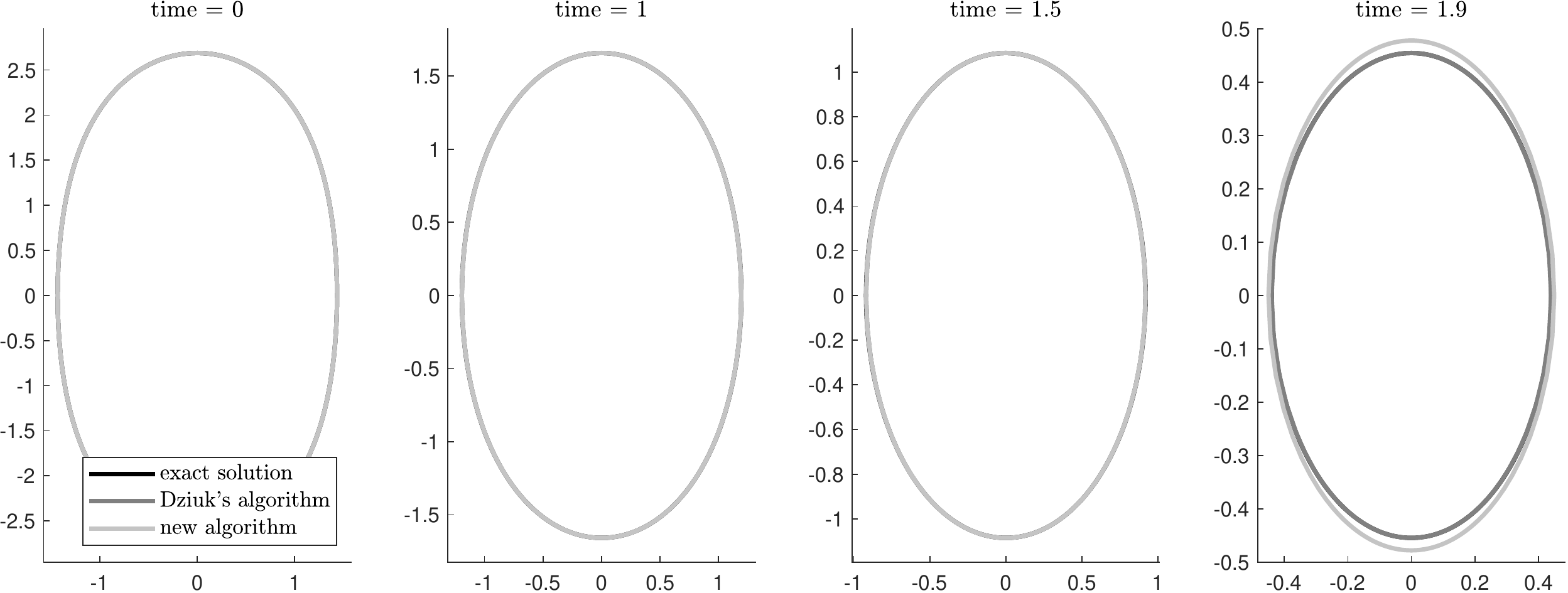}
	\caption{Comparing our algorithm (light grey) with the exact solution (black) and Dziuk's algorithm (grey) using (flat) Angenent ovals \eqref{eq:Angenent ovals}.}
	\label{fig:compare Angenent ovals}
\end{figure}

\subsection{Experiments for space curves}

We have performed various experiments for space curves as well, comparing our algorithm and Dziuk's.
In Figures~\ref{fig:trefoil_idem}--\ref{fig:snake} we report on the time evolution of a sinusoidal curve and a trefoil knot (which is eventually only immersed).

The numerical experiments in \cite{MCF,Willmore,MCF_generalised} have indicated that it is beneficial to conserve the geometric properties of the dynamic variables close to singularities, e.g.~for mean curvature flow projecting the extrapolated normal vector back to the unit sphere, cf.~\eqref{eq:extrapolation def}. 

According to our experiments the symmetry of $\pi_h$ is well preserved, however the idempotency $\pi_h^2 = \pi_h$ is deteriorating close to singularities.
Figure~\ref{fig:trefoil_idem} reports on an experiment where a (regularized) minimisation problem is solved (using Matlab's \texttt{fmincon}) in order to preserve idempotency, comparing it to the original algorithm. The regularisation step is performed only for those extrapolated projection matrices $\widetilde{\pi}_h^n$ \eqref{eq:extrapolation def} which are at least a tolerance away from being idempotent. That is a correction step, which is still locally Lipschitz, is only performed on the right-hand side of \eqref{eq:BDF}. (Finding an idempotent matrix close to $\widetilde{\pi}_h^n$ is a much harder problem then preserving unit length, cf.~\cite{MCF}, therefore this rudimentary process only yields a slight improvement.) 
In order to highlight this phenomena we used a coarse grid $\text{dof}=64$ and large step size $\tau = 0.01$ for Figure~\ref{fig:trefoil_idem}.
Such a geometric process is used for Figures~\ref{fig:trefoil} and \ref{fig:snake} as well.

\begin{figure}[htbp]
	\includegraphics[width=\textwidth,clip,trim={100 40 70 70}]
	{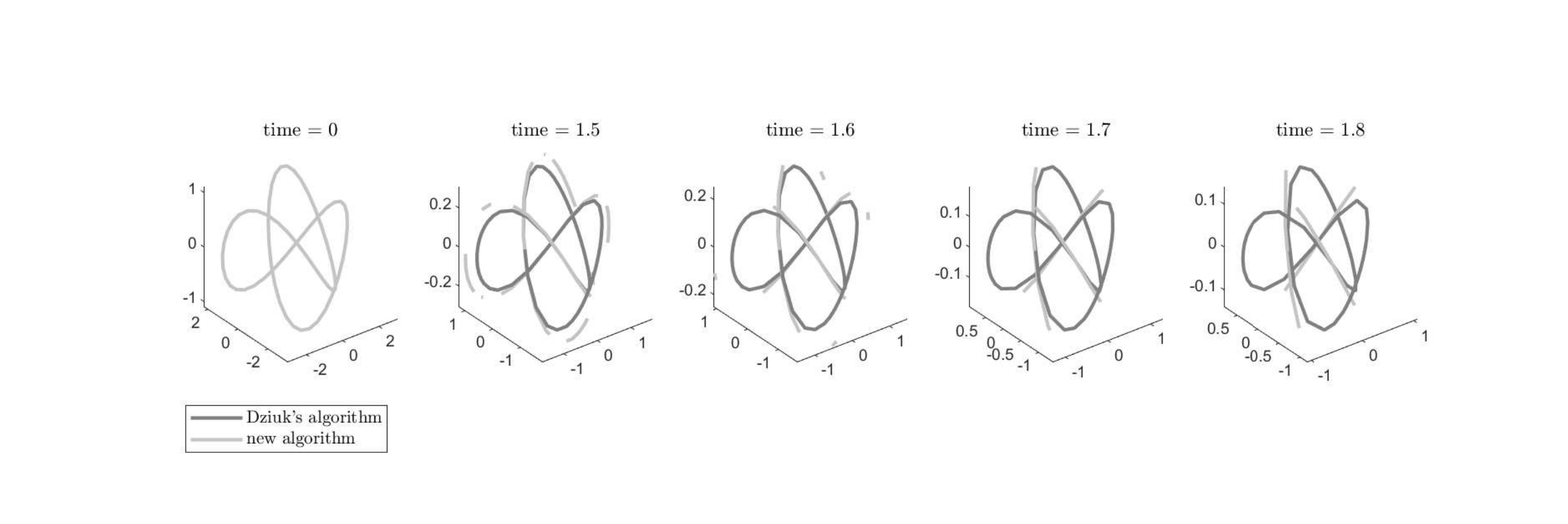}
	\includegraphics[width=\textwidth,clip,trim={100 40 70 70}]
	{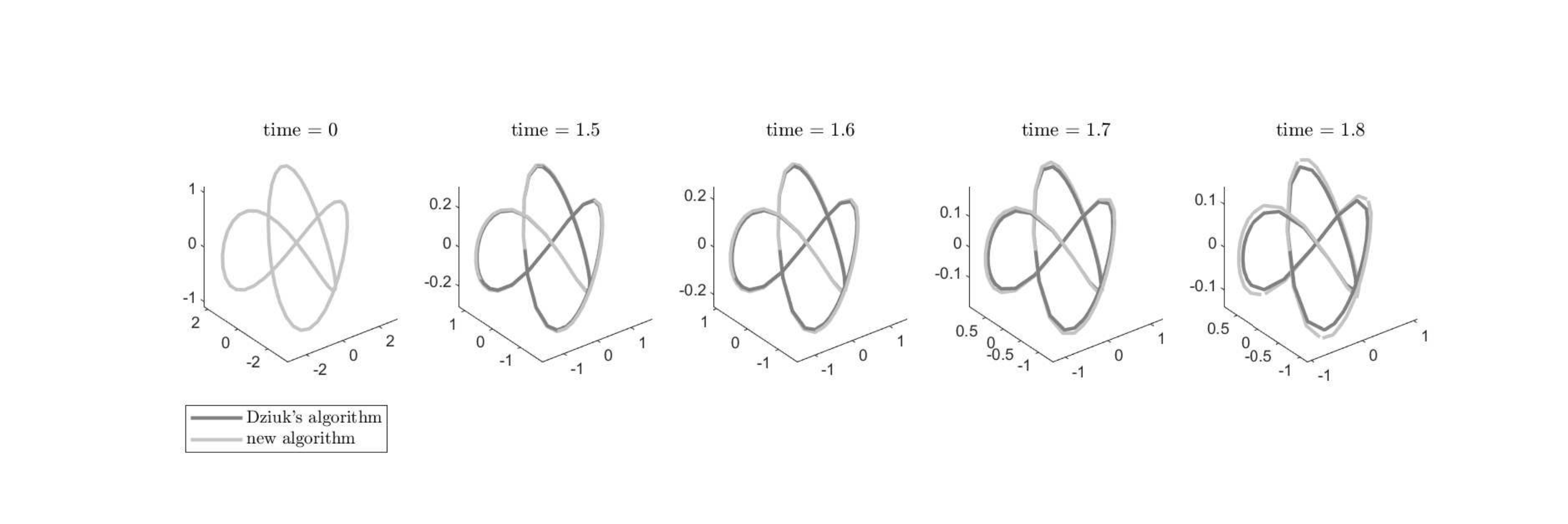}
	\caption{Comparing our algorithm (light grey) without and with idempotency correction (top and bottom), with \text{Dziuk's} algorithm (grey) using a trefoil knot as initial value. ($\text{dof}=64$ and $\tau = 0.01$)}
	\label{fig:trefoil_idem}
\end{figure}

\begin{figure}[htbp]
	\includegraphics[width=\textwidth,clip,trim={100 40 90 70}]
	{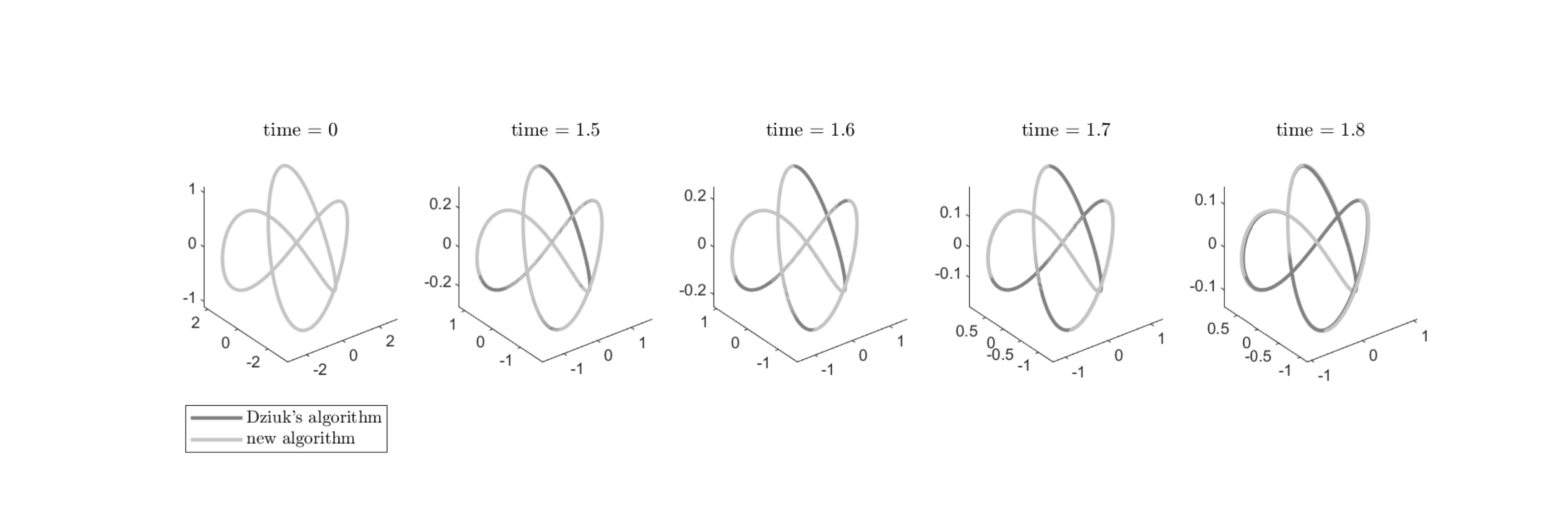}
	\caption{Comparing our algorithm (light grey) with \text{Dziuk's} algorithm (grey) using a trefoil knot as initial value. ($\text{dof}=512$ and $\tau = 10^{-4}$)}
	\label{fig:trefoil}
\end{figure}

\begin{figure}[htbp]
	\includegraphics[width=\textwidth,clip,trim={100 41 80 72}]
	{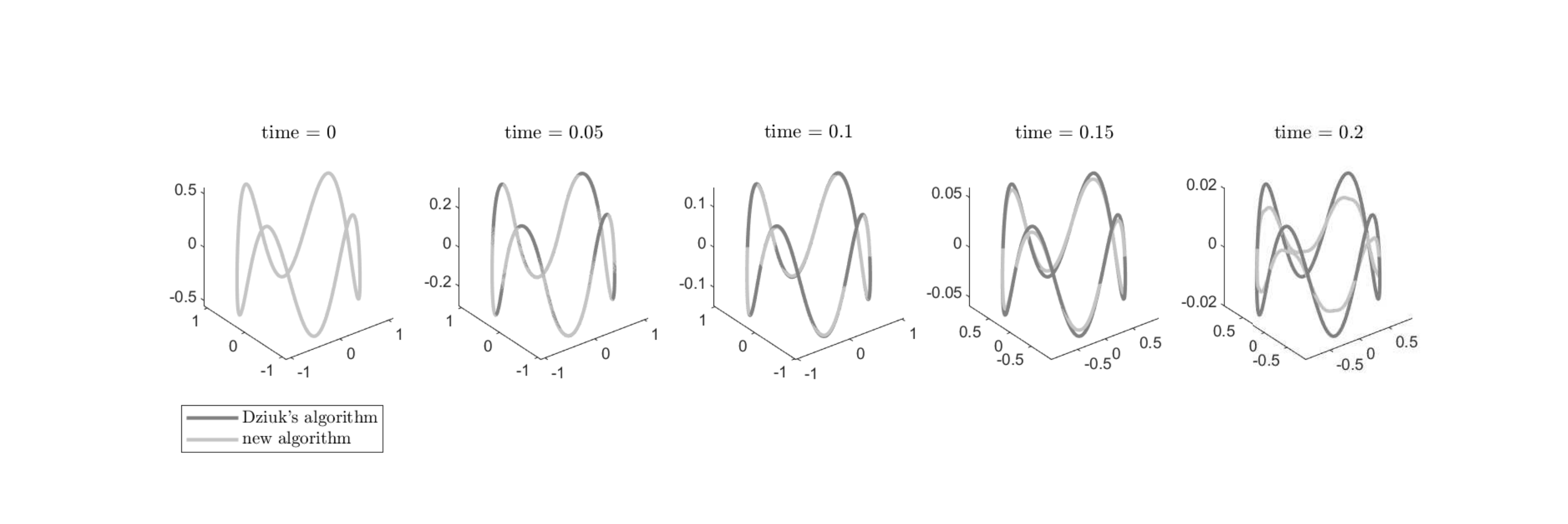}
	\caption{Comparing our algorithm (light grey) with \text{Dziuk's} algorithm (grey) using a sinusoidal initial curve. ($\text{dof}=512$ and $\tau = 10^{-4}$)}
	\label{fig:snake}
\end{figure}
For the evolution of the sinusoidal curve Figure~\ref{fig:snake}, we would like to highlight the short time scale, and the rapid shrinking in the $z$ coordinate.

%
%
%
%
%
%
%

\section*{Acknowledgments}

The authors wish to deeply thank Simon Brendle for bringing this topic to their attention, and also for his fundamental ideas, in particular, his contributions in deriving the evolution equations presented in the Appendix. 

The work of Bal\'azs Kov\'acs is funded by the Heisenberg Programme of the Deutsche Forschungsgemeinschaft (DFG, German Research Foundation) -- Project-ID 446431602.

\appendix
\section{Evolution equations for mean curvature flow in higher codimension}
\label{Appendix A}


Let $X \colon \Gamma^0 \times [0,T] \to \R^\dimR$ be a solution of the mean curvature flow, i.e.
\begin{equation*}
v = \vec{H} .
\end{equation*}
Let $g_{ij} 
= \partial_i X \cdot \partial_j X = \sum_{\mu} \partial_i X_{\mu} \partial_j X_{\mu}$ denote the induced metric, and $g^{ij}$ denote its inverse. 
Let 
\begin{equation*}
A_{ij} = (\partial_i \partial_j X)^{\bot} = \partial_i \partial_j X - \Ga_{ij}^k \partial_k X 
\end{equation*}
denote the second fundamental form
and 
\begin{equation*}
\vec{H} = g^{ij} A_{ij} = (g^{ij} \partial_i \partial_j X)^{\bot} = g^{ij} \partial_i \partial_j X - g^{ij} \Ga_{ij}^k \partial_k X 
\end{equation*}
the mean curvature vector. We shall view $\vec{H}$ as a function taking values in $\R^\dimR$. 
Further, let
\begin{equation*}
\pi = g^{ij} \partial_i X \otimes \partial_j X 
\end{equation*}
the orthogonal projection from $\R^\dimGa$ to the tangent space to the submanifold at the point $X(x,t)$. It can be seen as a function taking values in the space of $\dimR \times \dimR$-matrices.

\begin{lem}
	The evolution of the metric is given by 
	\[\frac{\partial}{\partial t} g_{ij} = -2 \, \vec{H} \cdot A_{ij}.\] 
	Moreover, the inverse metric satisfies
	\[\frac{\partial}{\partial t} g^{ij} = 2 \, g^{ik} \, g^{jl} \, \vec{H} \cdot A_{kl}.\] 
\end{lem}

\begin{proof} 
	We compute 
	\[\frac{\partial}{\partial t} g_{ij} = \partial_i \vec{H} \cdot \partial_j X + \partial_i X \cdot \partial_j \vec{H} = -\vec{H} \cdot \partial_i \partial_j X - \partial_j \partial_i X \cdot \vec{H} = -2 \, \vec{H} \cdot A_{ij}.\] 
	This proves the first statement. Since 
	\[\frac{\partial}{\partial t} g^{ij} = -g^{ik} \, g^{jl} \, \frac{\partial}{\partial t} g_{kl},\] 
	the second statement follows. 
\end{proof}

\begin{lem}
	\label{time.derivative.of.pi}
	We have 
	\[\frac{\partial}{\partial t} \pi = g^{ij} \, \partial_i \vec{H} \otimes \partial_j X + g^{ij} \, \partial_i X \otimes \partial_j \vec{H} + 2 \, g^{ik} \, g^{jl} \, (\vec{H} \cdot A_{kl}) \, \partial_i X \otimes \partial_j X.\] 
\end{lem}

\begin{proof} 
	This follows from the definition of $\pi$ together with the evolution equation of the metric. 
\end{proof}

\begin{lem} 
	\label{first.derivative.of.pi}
	The component-wise derivatives of $\pi$ are given by 
	\[\partial_k \pi = g^{ij} \, A_{ik} \otimes \partial_j X + g^{ij} \, \partial_i X \otimes A_{jk}.\]  
\end{lem}

\begin{proof}
	This follows from a direct calculation in geodesic normal coordinates. 
\end{proof}

\begin{lem} 
	\label{Laplacian.of.pi}
	The component-wise Laplacian of $\pi$ is given by 
	\begin{align*} 
	\Delta \pi 
	&= g^{ij} \, \partial_i \vec{H} \otimes \partial_j X + g^{ij} \, \partial_i X \otimes \partial_j \vec{H} + 2 \, g^{ip} \, g^{jq} \, (\vec{H} \cdot A_{ij}) \, \partial_p X \otimes \partial_q X \\ 
	&- 2 \, g^{ip} \, g^{jq} \, g^{kl} \, (A_{ik} \cdot A_{jl}) \, \partial_p X \otimes \partial_q X + 2 \, g^{ij} \, g^{kl} \, A_{ik} \otimes A_{jl}. 
	\end{align*} 
\end{lem}

\begin{proof} 
	Fix a point $p \in M$. We again work in geodesic normal coordinates around $p$. We compute 
	\begin{align*} 
	\Delta \pi 
	&= g^{ij} \, g^{kl} \, \partial_l A_{ik} \otimes \partial_j X + g^{ij} \, g^{kl} \, \partial_i X \otimes \partial_l A_{jk} \\ 
	&+ g^{ij} \, g^{kl} \, A_{ik} \otimes A_{jl} + g^{ij} \, g^{kl} \, A_{il} \otimes A_{jk} 
	\end{align*} 
	at the point $p$. Using the Codazzi equations, we obtain 
	\[(g^{kl} \, \partial_l A_{ik})^\perp = (\partial_i \vec{H})^\perp\] 
	at $p$. Moreover,
	\[\pi(g^{kl} \, \partial_l A_{ik}) = g^{kl} \, g^{pq} \, (\partial_l A_{ik} \cdot \partial_p X) \, \partial_q X = -g^{kl} \, g^{pq} \, (A_{ik} \cdot \partial_l \partial_p X) \, \partial_q X = -g^{kl} \, g^{pq} \, (A_{ik} \cdot A_{lp}) \, \partial_q X\]
	and 
	\[\pi(\partial_i \vec{H}) = g^{pq} \, (\partial_i \vec{H} \cdot \partial_p X) \, \partial_q X = -g^{pq} \, (\vec{H} \cdot \partial_i \partial_p X) \, \partial_q X = -g^{pq} \, (\vec{H} \cdot A_{ip}) \, \partial_q X\] 
	at $p$. Since $g^{kl} \, \partial_l A_{ik} = \pi(g^{kl} \, \partial_l A_{ik}) + (g^{kl} \, \partial_l A_{ik})^\perp$ and $\partial_i \vec{H} = \pi(\partial_i \vec{H}) + (\partial_i \vec{H})^\perp$, we conclude that 
	\[g^{kl} \, \partial_l A_{ik} = \partial_i \vec{H} + g^{pq} \, (\vec{H} \cdot A_{ip}) \, \partial_q X - g^{kl} \, g^{pq} \, (A_{ik} \cdot A_{lp}) \, \partial_q X\] 
	at $p$. Thus, 
	\begin{align*} 
	\Delta \pi 
	&= g^{ij} \, \partial_i \vec{H} \otimes \partial_j X + g^{ij} \, \partial_i X \otimes \partial_j \vec{H} \\ 
	&+ g^{ij} \, g^{pq} \, (\vec{H} \cdot A_{ip}) \, \partial_q X \otimes \partial_j X + g^{ij} \, g^{pq} \, (\vec{H} \cdot A_{jp}) \, \partial_i X \otimes \partial_q X \\ 
	&- g^{ij} \, g^{kl} \, g^{pq} \, (A_{ik} \cdot A_{lp}) \, \partial_q X \otimes \partial_j X - g^{ij} \, g^{kl} \, g^{pq} \, (A_{jk} \cdot A_{lp}) \, \partial_i X \otimes \partial_q X \\ 
	&+ g^{ij} \, g^{kl} \, A_{ik} \otimes A_{jl} + g^{ij} \, g^{kl} \, A_{il} \otimes A_{jk} 
	\end{align*} 
	at $p$. This proves the assertion. 
\end{proof}

\begin{lem}
	\label{heat.equation.for.pi}
	We have 
	\[\frac{\partial}{\partial t} \pi - \Delta \pi = 2 \, g^{ip} \, g^{jq} \, g^{kl} \, (A_{ik} \cdot A_{jl}) \, \partial_p X \otimes \partial_q X - 2 \, g^{ij} \, g^{kl} \, A_{ik} \otimes A_{jl},\] 
	where $\Delta \pi$ denotes the component-wise Laplacian.
\end{lem}

\begin{proof} 
	This follows from Lemma \ref{time.derivative.of.pi} and Lemma \ref{Laplacian.of.pi}. 
\end{proof}

In the following, Latin indices will run from $1$ to $\dimGa$, and Greek indices will run from $1$ to $\dimR$.

\begin{lem}
	\label{heat.equation.for.pi.v2}
	We have 
	\[\frac{\partial}{\partial t} \pi_{\alpha\beta} - \Delta \pi_{\alpha\beta} = 2 \sum_\mu g^{kl} \, \partial_k \pi_{\alpha\mu} \, \partial_l \pi_{\beta\mu} - 4 \sum_{\mu,\nu} g^{kl} \, \pi_{\mu\nu} \, \partial_k \pi_{\alpha\mu} \, \partial_l \pi_{\beta\nu}.\] 
\end{lem}

\begin{proof} 
	We compute
	\[\sum_\mu g^{kl} \, \partial_k \pi_{\alpha\mu} \, \partial_l \pi_{\beta\mu} = [g^{ij} \, g^{kl} \, A_{ik} \otimes A_{jl} + g^{ip} \, g^{jq} \, g^{kl} \, (A_{ik} \cdot A_{jl}) \, \partial_p X \otimes \partial_q X]_{\alpha\beta}\] 
	and 
	\[\sum_{\mu,\nu} g^{kl} \, \pi_{\mu\nu} \, \partial_k \pi_{\alpha\mu} \, \partial_l \pi_{\beta\nu} = [g^{ij} \, g^{kl} \, A_{ik} \otimes A_{jl}]_{\alpha\beta}.\] 
	Hence, the assertion follows from Lemma \ref{heat.equation.for.pi}. 
\end{proof}

Finally, let us derive the evolution equation for the mean curvature vector $\vec{H}$.

\begin{lem}
	\label{heat.equation.for.H}
	The evolution of the mean curvature is given by 
	\[\frac{\partial}{\partial t} \vec{H} - \Delta \vec{H} = 2 \, g^{ik} \, g^{jl} \, (\vec{H} \cdot A_{kl}) \, A_{ij} + 2\, g^{ij} \, g^{kl} \, (\partial_i \vec{H} \cdot A_{jl}) \, \partial_k X,\]
	where $\Delta \vec{H}$ denotes the component-wise Laplacian.
\end{lem} 

\begin{proof} 
	The mean curvature vector is given by  
	\[\vec{H} = g^{ij} \, \partial_i \partial_j X - g^{ij} \, \Gamma_{ij}^k \, \partial_k X\] 
	at each point in space-time. Let us fix a point $p$ and work in geodesic normal coordinates around $p$. In particular, $\Gamma_{ij}^k = 0$ at $p$. At the point $p$, we have 
	\[\frac{\partial}{\partial t} \vec{H} = g^{ij} \, \partial_i \partial_j (\frac{\partial}{\partial t} X) + \frac{\partial}{\partial t}(g^{ij}) \, \partial_i \partial_j X - g^{ij} \, (\frac{\partial}{\partial t} \Gamma_{ij}^k) \, \partial_k X.\] 
	This implies 
	\[\frac{\partial}{\partial t} \vec{H} - \Delta \vec{H} = 2 \, g^{ik} \, g^{jl} \, (\vec{H} \cdot A_{kl}) \, A_{ij} - g^{ij} \, (\frac{\partial}{\partial t} \Gamma_{ij}^k) \, \partial_k X\] 
	at $p$. We next compute 
	\begin{align*} 
	\frac{\partial}{\partial t} \Gamma_{ij}^k 
	&= \frac{1}{2} \, g^{kl} \, (\partial_i \frac{\partial}{\partial t} g_{jl} + \partial_j \frac{\partial}{\partial t} g_{il} - \partial_l \frac{\partial}{\partial t} g_{ij}) \\ 
	&= -g^{kl} \, (\partial_i (\vec{H} \cdot A_{jl}) + \partial_j (\vec{H} \cdot A_{il}) - \partial_l (\vec{H} \cdot A_{ij})) \\ 
	&= -g^{kl} \, ((\partial_i \vec{H} \cdot A_{jl}) + (\vec{H} \cdot \partial_i A_{jl})) \\ 
	&- g^{kl} \, ((\partial_j \vec{H} \cdot A_{il}) + (\vec{H} \cdot \partial_j A_{il})) \\ 
	&+ g^{kl} \, ((\partial_l \vec{H} \cdot A_{ij}) + (\vec{H} \cdot \partial_l A_{ij})) 
	\end{align*} 
	at $p$. Consequently, 
	\[g^{ij} \, \frac{\partial}{\partial t} \Gamma_{ij}^k = -2 \, g^{ij} \, g^{kl} \, \partial_i \vec{H} \cdot A_{jl} - 2 \, g^{ij} \, g^{kl} \, \vec{H} \cdot \partial_i A_{jl} + 2 \, g^{kl} \, \vec{H} \cdot \partial_l \vec{H}\]
	at $p$. Using the Codazzi equations, we obtain $(g^{ij} \, \partial_i A_{jl})^\perp = (\partial_l \vec{H})^\perp$ at $p$, hence 
	\[g^{ij} \, \frac{\partial}{\partial t} \Gamma_{ij}^k = -2 \, g^{ij} \, g^{kl} \, \partial_i \vec{H} \cdot A_{jl}\] 
	at $p$. Putting these facts together, we conclude that 
	\[\frac{\partial}{\partial t} \vec{H} - \Delta \vec{H} = 2 \, g^{ik} \, g^{jl} \, (\vec{H} \cdot A_{kl}) \, A_{ij} + 2\, g^{ij} \, g^{kl} \, (\partial_i \vec{H} \cdot A_{jl}) \, \partial_k X\]
	at $p$. This proves the assertion. 
\end{proof}

\begin{lem}
	\label{heat.equation.for.H.v2}
	The evolution of the mean curvature is given by 
	\[\frac{\partial}{\partial t} \vec{H}_\alpha - \Delta \vec{H}_\alpha = 2 \sum_\beta g^{kl} \, \partial_k \pi_{\alpha\beta} \, \partial_l \vec{H}_\beta + 4 \sum_{\beta,\mu} g^{kl} \, \partial_k \pi_{\alpha\mu} \, \partial_l \pi_{\beta\mu} \, \vec{H}_\beta,\]
	where $\Delta \vec{H}$ denotes the component-wise Laplacian.
\end{lem} 

\begin{proof}
	The identity 
	\[\sum_\mu g^{kl} \, \partial_k \pi_{\alpha\mu} \, \partial_l \pi_{\beta\mu} = [g^{ij} \, g^{kl} \, A_{ik} \otimes A_{jl} + g^{ip} \, g^{jq} \, g^{kl} \, (A_{ik} \cdot A_{jl}) \, \partial_p X \otimes \partial_q X]_{\alpha\beta}\] 
	gives 
	\[\sum_{\beta,\mu} g^{kl} \, \partial_k \pi_{\alpha\mu} \, \partial_l \pi_{\beta\mu} \, \vec{H}_\beta = [g^{ij} \, g^{kl} \, (\vec{H} \cdot A_{jl}) \, A_{ik}]_\alpha.\] 
	Moreover, using the identity $\partial_l \vec{H} \cdot \partial_j X = -\vec{H} \cdot \partial_l \partial_j X = -\vec{H} \cdot A_{jl}$, we obtain 
	\begin{align*} 
	\sum_\beta g^{kl} \, \partial_k \pi_{\alpha\beta} \, \partial_l \vec{H}_\beta 
	&= [g^{ij} \, g^{kl} \, (\partial_l \vec{H} \cdot A_{jk}) \, \partial_i X + g^{ij} \, g^{kl} \, (\partial_l \vec{H} \cdot \partial_j X) \, A_{ik}]_\alpha \\ 
	&= [g^{ij} \, g^{kl} \, (\partial_l \vec{H} \cdot A_{jk}) \, \partial_i X - g^{ij} \, g^{kl} \, (\vec{H} \cdot A_{jl}) \, A_{ik}]_\alpha. 
	\end{align*} 
	Hence, the assertion follows from Lemma \ref{heat.equation.for.H}. 
\end{proof}

\bibliographystyle{siamplain}
\bibliography{MCF_codim_literature}

\begin{thebibliography}{10}

\bibitem{Akrivisetal_BDF6}
{\sc G.~Akrivis, M.~Chen, F.~Yu, and Z.~Zhou}, {\em The energy technique for
  the six-step {BDF} method}, arXiv:2007.08924,  (2020).

\bibitem{AkrivisLiLubich_quasilinBDF}
{\sc G.~Akrivis, B.~Li, and C.~Lubich}, {\em Combining maximal regularity and
  energy estimates for time discretizations of quasilinear parabolic
  equations}, Math. Comp., 86 (2017), pp.~1527--1552.

\bibitem{Altschuler}
{\sc S.~J. Altschuler}, {\em Singularities of the curve shrinking flow for
  space curves}, J. Differential Geom., 34 (1991), pp.~491--514.

\bibitem{AltschulerGrayson}
{\sc S.~J. Altschuler and M.~A. Grayson}, {\em Shortening space curves and flow
  through singularities}, J. Differential Geom., 35 (1992), pp.~283--298.

\bibitem{AmbrosioSoner_1}
{\sc L.~Ambrosio and H.~M. Soner}, {\em Flow by mean curvature of surfaces of
  any codimension}, in Variational methods for discontinuous structures
  ({C}omo, 1994), vol.~25 of Progr. Nonlinear Differential Equations Appl.,
  Birkh\"{a}user, Basel, 1996, pp.~123--134.

\bibitem{AmbrosioSoner_levelset}
{\sc L.~Ambrosio and H.~M. Soner}, {\em Level set approach to mean curvature
  flow in arbitrary codimension}, J. Differential Geom., 43 (1996),
  pp.~693--737.

\bibitem{AndrewsBaker}
{\sc B.~Andrews and C.~Baker}, {\em Mean curvature flow of pinched submanifolds
  to spheres}, J. Differential Geom., 85 (2010), pp.~357--395.

\bibitem{Angenent_ovals}
{\sc S.~Angenent}, {\em Formal asymptotic expansions for symmetric ancient
  ovals in mean curvature flow}, Netw. Heterog. Media, 8 (2013), pp.~1--8.

\bibitem{BGN_curve_grad_flow}
{\sc J.~W. Barrett, H.~Garcke, and R.~N\"{u}rnberg}, {\em Numerical
  approximation of gradient flows for closed curves in {$\Bbb R^d$}}, IMA J.
  Numer. Anal., 30 (2010), pp.~4--60.

\bibitem{BGN_curves_2012}
{\sc J.~W. Barrett, H.~Garcke, and R.~N\"{u}rnberg}, {\em Parametric
  approximation of isotropic and anisotropic elastic flow for closed and open
  curves}, Numer. Math., 120 (2012), pp.~489--542.

\bibitem{MCF_generalised}
{\sc T.~Binz and B.~Kov{\'a}cs}, {\em A convergent finite element algorithm for
  generalized mean curvature flows of closed surfaces}, \textnormal{to appear
  in} IMA Journal of Numerical Analysis,  (2021).
\newblock doi.org/10.1093/imanum/drab043.

\bibitem{CarliniFalconeFerretti_CSF_codim_2}
{\sc E.~Carlini, M.~Falcone, and R.~Ferretti}, {\em A semi-{L}agrangian scheme
  for the curve shortening flow in codimension-2}, J. Comput. Phys., 225
  (2007), pp.~1388--1408.

\bibitem{DeckelnickDziuk1994}
{\sc K.~Deckelnick and G.~Dziuk}, {\em On the approximation of the curve
  shortening flow}, in Calculus of variations, applications and computations
  ({P}ont-\`a-{M}ousson, 1994), vol.~326 of Pitman Res. Notes Math. Ser.,
  Longman Sci. Tech., Harlow, 1995, pp.~100--108.

\bibitem{DeckelnickDziukElliott_acta}
{\sc K.~Deckelnick, G.~Dziuk, and C.~M. Elliott}, {\em Computation of geometric
  partial differential equations and mean curvature flow}, Acta Numer., 14
  (2005), pp.~139--232.

\bibitem{Demlow2009}
{\sc A.~Demlow}, {\em Higher--order finite element methods and pointwise error
  estimates for elliptic problems on surfaces}, SIAM J. Numer. Anal., 47
  (2009), pp.~805--807.

\bibitem{DoerflerNuernberg}
{\sc W.~D\"{o}rfler and R.~N\"{u}rnberg}, {\em Discrete gradient flows for
  general curvature energies}, SIAM J. Sci. Comput., 41 (2019),
  pp.~A2012--A2036.

\bibitem{Dziuk88}
{\sc G.~Dziuk}, {\em Finite elements for the {B}eltrami operator on arbitrary
  surfaces}, Partial differential equations and calculus of variations, Lecture
  Notes in Math., 1357, Springer, Berlin,  (1988), pp.~142--155.

\bibitem{Dziuk_CSF_1994}
{\sc G.~Dziuk}, {\em Convergence of a semi-discrete scheme for the curve
  shortening flow}, Math. Models Methods Appl. Sci., 4 (1994), pp.~589--606.

\bibitem{DziukElliott_ESFEM}
{\sc G.~Dziuk and C.~Elliott}, {\em Finite elements on evolving surfaces}, IMA
  J. Numer. Anal., 27 (2007), pp.~262--292.

\bibitem{DziukElliott_acta}
{\sc G.~Dziuk and C.~Elliott}, {\em Finite element methods for surface {PDE}s},
  Acta Numerica, 22 (2013), pp.~289--396.

\bibitem{HairerWannerII}
{\sc E.~Hairer and G.~Wanner}, {\em Solving Ordinary Differential Equations II.
  Stiff and Differential--Algebraic Problems}, Springer, Berlin, {S}econd~ed.,
  1996.

\bibitem{Huisken1984}
{\sc G.~Huisken}, {\em Flow by mean curvature of convex surfaces into spheres},
  J. Differential Geometry, 20 (1984), pp.~237--266.

\bibitem{highorderESFEM}
{\sc B.~Kov{\'a}cs}, {\em High-order evolving surface finite element method for
  parabolic problems on evolving surfaces}, IMA J. Numer. Anal., 38 (2018),
  pp.~430--459.

\bibitem{MCF}
{\sc B.~Kov{\'a}cs, B.~Li, and C.~Lubich}, {\em A convergent evolving finite
  element algorithm for mean curvature flow of closed surfaces}, Numer. Math.,
  143 (2019), pp.~797--853.

\bibitem{MCF_soldriven}
{\sc B.~Kov{\'a}cs, B.~Li, and C.~Lubich}, {\em A convergent algorithm for
  forced mean curvature flow driven by diffusion on the surfaces}, Interfaces
  Free Bound., 22 (2020), pp.~443--464.

\bibitem{Willmore}
{\sc B.~Kov{\'a}cs, B.~Li, and C.~Lubich}, {\em A convergent evolving finite
  element algorithm for {W}illmore flow of closed surfaces},  (2020).
\newblock arXiv:2007.15257.

\bibitem{KLLP2017}
{\sc B.~Kov\'{a}cs, B.~Li, C.~Lubich, and C.~{Power Guerra}}, {\em Convergence
  of finite elements on an evolving surface driven by diffusion on the
  surface}, Numer. Math., 137 (2017), pp.~643--689.

\bibitem{LubichMansourVenkataraman_bdsurf}
{\sc C.~Lubich, D.~Mansour, and C.~Venkataraman}, {\em Backward difference time
  discretization of parabolic differential equations on evolving surfaces}, IMA
  J. Numer. Anal., 33 (2013), pp.~1365--1385.

\bibitem{LynchNguyen}
{\sc S.~Lynch and H.~Nguyen}, {\em Pinched ancient solutions to the high
  codimension mean curvature flow}, Calc. Var., 60 (2021).

\bibitem{MikulaUrban}
{\sc K.~Mikula and J.~Urb\'{a}n}, {\em A new tangentially stabilized 3{D} curve
  evolution algorithm and its application in virtual colonoscopy}, Adv. Comput.
  Math., 40 (2014), pp.~819--837.

\bibitem{Naff}
{\sc K.~Naff}, {\em A planarity estimate for pinched solutions of mean
  curvature flow},  (2019).
\newblock arXiv:{1906.08184}.

\bibitem{Pozzi_CSF_codim}
{\sc P.~Pozzi}, {\em Anisotropic curve shortening flow in higher codimension},
  Math. Methods Appl. Sci., 30 (2007), pp.~1243--1281.

\bibitem{Pozzi_surf_codim}
{\sc P.~Pozzi}, {\em Anisotropic mean curvature flow for two-dimensional
  surfaces in higher codimension: a numerical scheme}, Interfaces Free Bound.,
  10 (2008), pp.~539--576.

\bibitem{Smoczyk_survey}
{\sc K.~Smoczyk}, {\em Mean curvature flow in higher codimension: introduction
  and survey}, in Global differential geometry, vol.~17 of Springer Proc.
  Math., Springer, Heidelberg, 2012, pp.~231--274.

\bibitem{Wang_MCF_codim}
{\sc M.-T. Wang}, {\em Mean curvature flows in higher codimension}, in Second
  {I}nternational {C}ongress of {C}hinese {M}athematicians, vol.~4 of New Stud.
  Adv. Math., Int. Press, Somerville, MA, 2004, pp.~275--283.

\bibitem{Wang_lecturenotes}
{\sc M.-T. Wang}, {\em Lectures on mean curvature flows in higher
  codimensions}, in Handbook of geometric analysis. {N}o. 1, vol.~7 of Adv.
  Lect. Math. (ALM), Int. Press, Somerville, MA, 2008, pp.~525--543.

\end{thebibliography}

\end{document}